\documentclass[12pt]{amsart}
\usepackage{amsfonts}
\usepackage{amssymb, amscd, amsthm}
\usepackage{latexsym}
\usepackage{multirow}
\usepackage{amsmath}
\usepackage[all]{xy}
\usepackage{mathrsfs}
\usepackage{amsbsy}
\usepackage{amsfonts}
\usepackage{mathrsfs}
\usepackage{verbatim}
\usepackage{version}
\usepackage{color}

\newtheorem{theorem}{Theorem}[section]
\newtheorem{lemma}[theorem]{Lemma}
\newtheorem{thm}[theorem]{Theorem}
\newtheorem{prop}[theorem]{Proposition}
\newtheorem{rem}[theorem]{Remark}
\newtheorem{coro}[theorem]{Corollary}
\newtheorem{defn}[theorem]{Definition}

\newtheorem{con/que}[theorem]{Conjecture/Question}

\newcommand{\ra}{\rightarrow}
\newcommand{\mo}{\mathcal{O}}
\newcommand{\mf}{\mathcal{F}}
\newcommand{\mg}{\mathcal{G}}
\newcommand{\ma}{\mathcal{A}}

\newcommand{\me}{\mathcal{E}}
\newcommand{\mi}{\mathcal{I}}
\newcommand{\mk}{\mathcal{K}}
\newcommand{\mext}{\mathbb{E}\mathtt{x}\mathtt{t}}

\newcommand{\mt}{\mathcal{T}}

\newcommand{\mn}{\mathcal{N}}

\newcommand{\cl}{\mathcal{L}}
\newcommand{\ts}{\textbf{S}}
\newcommand{\tb}{\mathtt{B}}
\newcommand{\tw}{\textbf{W}}
\newcommand{\mc}{\mathcal{C}}
\newcommand{\mr}{\mathcal{R}}

\newcommand{\E}{\mathscr{E}}
\newcommand{\F}{\mathscr{F}}
\newcommand{\G}{\mathscr{G}}

\newcommand{\hh}{\mathscr{H}}

\newcommand{\A}{\mathscr{A}}
\newcommand{\B}{\mathscr{B}}

\newcommand{\I}{\mathscr{I}}

\newcommand{\Hom}{\operatorname{Hom}}
\newcommand{\Ext}{\operatorname{Ext}}
\newcommand{\Pic}{\operatorname{Pic}}

\newcommand{\Tor}{\operatorname{Tor}}

\def\<{\langle}
\def\>{\rangle}

\newcommand{\z}{\Theta}

\newcommand{\dzd}{D_{\Theta_{dH}}}

\newcommand{\ml}{M(L,0)}

\newcommand{\md}{M(dH,0)}

\newcommand{\wrn}{W(r,0,n)}
\newcommand{\wrr}{W(r,0,r)}

\newcommand{\kwrn}{\mathfrak{W}(r,0,n)}

\newcommand{\lcl}{\lambda_{c^r_n}(L)}

\newcommand{\lrl}{\lambda_r(L)}

\newcommand{\crn}{c^r_n}

\newcommand{\cb}{\textbf{CB}}

\newcommand{\ls}{|L|}
\newcommand{\p}{\mathbb{P}}
\newcommand{\bz}{\mathbb{Z}}
\newcommand{\bc}{\mathbb{C}}
\newcommand{\km}{\mathfrak{M}}
\newcommand{\kw}{\mathfrak{W}}
\newcommand{\kp}{\mathfrak{P}}
\newcommand{\ks}{\mathfrak{S}}
\newcommand{\ke}{\mathfrak{E}}
\newcommand{\kz}{\mathfrak{Z}}

\begin{document}
\fontsize{12pt}{14pt} \textwidth=14cm \textheight=21 cm
\numberwithin{equation}{section}
\title{Strange duality on rational surfaces II: higher rank cases.}
\author{Yao Yuan}
\address{Yau Mathematical Sciences Center, Tsinghua University, 100084, Beijing, P. R. China}
\email{yyuan@mail.tsinghua.edu.cn; yyuan@math.tsinghua.edu.cn}
\subjclass[2000]{Primary 14D05}

\begin{abstract} We study Le Potier's strange duality conjecture on a rational surface.  We focus on the strange duality map $SD_{c_n^r,L}$ which involves the moduli space of rank $r$ sheaves with trivial first Chern class and second Chern class $n$, and the moduli space of 1-dimensional sheaves with determinant $L$ and Euler characteristic 0.   We show there is an exact sequence relating the map $SD_{c_r^r,L}$ to $SD_{c^{r-1}_{r},L}$ and $SD_{c_r^r,L\otimes K_X}$ for all $r\geq1$ under some conditions on $X$ and $L$ which applies to a large number of cases on $\p^2$ or Hirzebruch surfaces .
Also on $\mathbb{P}^2$ 
we show that for any $r>0$, $SD_{c^r_r,dH}$ is an isomorphism for $d=1,2$, injective for $d=3$ and moreover $SD_{c_3^3,rH}$ and $SD_{c_3^2,rH}$ are injective.   At the end we prove that the map $SD_{c_n^2,L}$ ($n\geq2$) is an isomorphism for $X=\mathbb{P}^2$ or Fano rational ruled surfaces and $g_L=3$, and hence so is $SD_{c_3^3,L}$ as a corollary of our main result.   

\end{abstract}
\maketitle
\tableofcontents
\section{Introduction.}
\subsection{History \& Set-up}\qquad

Strange duality conjecture was at first formulated for moduli spaces of vector bundles over curves (see \cite{Bea},\cite{DT}) and has been proved about ten years before (\cite{Bel1},\cite{Bel2},\cite{MO1}).  Under some suitable conditions, this conjecture can also be formulated for moduli spaces of semistable sheaves over surfaces.  However, there is no general extension to surfaces so far.  Mainly there are two formulations for surfaces, one of which is due to Le Potier (see \cite{LPst}, \cite{Da2} or \S2.4 in \cite{GY}) for simply connected surfaces, while the other is due to Marian-Oprea for K3 and Abelian surfaces (see \cite{MO2} or \cite{MO4}).  Both formulations have been studied by many people and the conjecture has been proven true for a number of cases (\cite{Abe},\cite{Abe2},\cite{BMOY},\cite{Da1},\cite{Da2},\cite{GY},\cite{MO3},\cite{MO4},\cite{MO5},\cite{Yuan1},\cite{Yuan5},\cite{Yuan6}).  In spite of that, on strange duality for surfaces what we have known is still little. 

In this paper, we will work on Le Potier's strange duality conjecture.  Let us briefly review the set-up.  More details can be found in \cite{Da2}, \cite{LPst}, \cite{MO2} or \S2 in \cite{GY}. 

Let $X$ be any smooth projective scheme over $\bc$.  Let $u$ and $c$ be two elements in the Grothendieck group $K(X)$ of coherent sheaves on $X$, assume moreover $u$ is orthogonal to $c$ with respect to the Euler characteristic, i.e. the flat tensor $\mf_u\otimes^{L} \mf_c$ is of Euler characteristic zero for any $\mf_u$ ($\mf_c$, resp.) a sheaf in class $u$ ($c$, resp.).  Denote by $M(u)$ ($M(c)$, resp.) the moduli space of semistable sheaves of class $u$ ($c$, resp.).  We ask the determinant line bundle $\lambda_u(c)$ ($\lambda_c(u)$, resp.) associated to $c$ ($u$, resp.) on $M(u)$ ($M(c)$, resp.) is well-defined.  Notice that if there are strictly semistable sheaves, we will need a slightly stronger condition than $\chi(\mf_u\otimes^{L} \mf_c)=0$ to define $\lambda_u(c)$ and $\lambda_c(u)$.  We refer to \S 2 in \cite{GY} or Chapter 8 in \cite{HL} for the explicit definition of determinant line bundles.  The definition in \cite{HL} is dual to what we use in this paper.

The locus $\mathscr{D}_{c,u}:=\{(\mf_c,\mf_u)\in M(c)\times M(u)|H^0(\mf_c\otimes \mf_u)\neq 0\}$ is closed in $M(c)\times M(u)$.  If $\mathscr{D}$ is a divisor of the line bundle $\lambda_c(u)\boxtimes\lambda_u(c)$ (not always the case on surfaces), then the section induced by $\mathscr{D}_{c,u}$ defines the following \emph{strange duality map} up to scalars.
\begin{equation}\label{map}SD_{c,u}:H^0(M(c),\lambda_c(u))^{\vee}\ra H^0(M(u),\lambda_u(c)).\end{equation}
Strange duality conjecture says that $SD_{c,u}$ is an isomorphism.

In Le Potier's formulation (\cite{LPst} p.9), the condition ($\bigstar$) as follows is satisfied, which assures that $\mathscr{D}_{c,u}$ is a divisor of the line bundle $\lambda_c(u)\boxtimes\lambda_u(c)$ and hence the map $SD_{c,u}$ can be defined.

($\bigstar $) \emph{for all semistable sheaves $\mf$ of class $c$ and semistable sheaves $\mg$ of class $u$ on $X$, $\Tor^i(\mf,\mg)=0$ $\forall~i\ge 1$, and $H^2(X,\mf\otimes \mg)=0$.}

In this paper, we let $X$ be a rational surface over $\mathbb{C}$ and specify $u=u_{L}$ and $c=c^r_n$ (def. see \S\ref{NaP} (4) (5)).  It is easy to check that ($\bigstar$) is fulfilled.  We want to study whether $SD_{c,u}$ in (\ref{map}) is an isomorphism.  We also write $SD_{c^r_n,L}$ for our specified $c=c^r_n$ and $u=u_L$, in particular $SD_{r,L}:=SD_{c_r^r,L}$.  

\subsection{Results.}\qquad

Our results are of two parts.  In the first part, we construct a bridge from maps $SD_{c_r^{r-1},L}$ and $SD_{r,L\otimes K_X}$ to $SD_{r,L}$.  The main result contains Proposition \ref{hrmain} and Proposition \ref{beta1}, and we prove the following three theorems as applications to our main result.
\begin{thm}[Corollary \ref{beta2}]\label{introbeta2}Let $(X,L)$ be one of the following cases.
\begin{enumerate}
\item $X=\p^2$, $L=dH$ for $d>0$.
\item $X=\p(\mo_{\p^1}\oplus\mo_{\p^1}(e)):=\Sigma_e$ with $F$ the fiber class and $G$ the section such that $G.G=-e$, then 
$L=aG+bF$ are one of the following
\begin{itemize}
\item $\min\{a,b\}\leq 1$;
\item $\min\{a,b\}\geq 2$, $e\neq 1$, $L$ ample;
\item $\min\{a,b\}\geq 2$, $e=1$, $b\geq a+[a/2]$ with $[a/2]$ the integral part of $a/2$.
\end{itemize}
\end{enumerate}
Then we have for all $r\geq2$
\[\left.\begin{array}{r}SD_{c_r^{r-1},L}\text{ is injective (surjective, an isomorphism, resp.)}\\
SD_{r,L\otimes K_X}\text{ is injective (surjective, an isomorphism, resp.)}\end{array}\right\}\Rightarrow
\text{So is }SD_{r,L}.\] 
\end{thm}

\begin{thm}[Theorem \ref{rank3}]\label{introrank3}Let $(X,L)$ be as in Theorem \ref{introbeta2} and let $r=3=n$.  If $H^0(L\otimes K_X)=0$, then
\begin{enumerate}\item $H^0(W(3,0,3),\lambda_{3}(L\otimes K_X))=0$;

\item $SD_{3,L}$ is an isomorphism.
\end{enumerate}
\end{thm}

\begin{thm}[Theorem \ref{rank3a}]\label{introrank3a}Let $(X,L)$ be as in Theorem \ref{introbeta2} and let $r=3=n$.  $X$ and $L$ be as follows.
\begin{enumerate}
\item $X=\p^2$ or $\Sigma_e$ with $e=0,1$.  $L=-K_X$.

\item $X=\Sigma_e$ with $e=0,1$. $L=-K_X+F$ with $F$ the fiber class.
\end{enumerate}
Then $SD_{3,L}$ is an isomorphism.
\end{thm}

Especially for $X=\p^2$, the results can be improved as follows.
\begin{thm}[Theorem \ref{p2}]\label{introp2}Let $X=\p^2$, $L=dH$.  Then
\begin{enumerate}
\item $SD_{r,dH}$ is an isomorphism for $d=1,2$ and $r>0$;
\item $SD_{c_r^{r-1},dH}$ is an isomorphism for $d=1,2$ and $r>1$;
\item $SD_{3,rH}$ is injective for $r>0$;
\item $SD_{c_3^2,rH}$ is injective for all $r>0$. 
\end{enumerate}
\end{thm}

In the second part, we let $c=c_n^2$ and prove the following two theorems.

\begin{thm}[Theorem \ref{gth3}]\label{introgth3}Let $(X,L)$ be as follows.
\begin{enumerate}\item $X=\p^2$ and $L=4H$.  
\item $X=\Sigma_e:=\p(\mo_{\p^1}\oplus\mo_{\p}(e))$ with $e\leq 1$ and $L=2G+(e+4)F$ where $F$ is the fiber class and $G$ is the section class such that $G.G=-e$.  
\end{enumerate}
Then under suitable polarizations $SD_{c_n^2,L}$ is an isomorphism for any $n\geq3$.
\end{thm}

\begin{thm}[Corollary \ref{gth5}]\label{introgth5}Let $(X,L)$ be as in Theorem \ref{introgth3}.  Then under suitable polarizations $SD_{3,L}$ is an isomorphism.
\end{thm}

Theorem \ref{introgth5} is just a corollary to Theorem \ref{introgth3}, Theorem \ref{introrank3} and Theorem \ref{introbeta2}.

The strategy for the first part is to find a divisor $\ts_r\subset M(c^r_r)$ of $\lambda_{c_r^r}(u_{K_X^{-1}})$ and construct a birational morphism $\delta:M(c_{r}^{r-1})\ra\ts_r$ such that $\delta^{*}\lambda_{c_r^r}(u_L)\cong \lambda_{c_r^{r-1}}(u_L)$ and $\delta_{*}\mo_{M(c_r^{r-1})}\cong \mo_{\ts_r}$.  This generalizes the key idea in \cite{Yuan6}.

The strategy for the second part is at first to show the equation \begin{equation}\label{wanna}h^0(M(c_n^2),\lambda_{c^2_n}(u_L))=h^0(M(u_L),\lambda_{u_L}(c_n^2)),\end{equation} and then to show the surjectivity of $SD_{c^2_n,L}$.  

The LHS of (\ref{wanna}) is also equal to $\chi(M(c_n^2),\lambda_{c^2_n}(u_L))$ and has been computed in \cite{GY} for $X=\Sigma_e~(e=0,1)$ and in \cite{G1} for $X=\p^2$.  So we only need to compute $h^0(M(u_L),\lambda_{u_L}(c_n^2))$.  To show the surjectivity of $SD_{c^2_n,L}$, we find enough $\mg_i\in M(c_n^2)$ such that they induce sections $s_{\mg_i}$ of $\lambda_{u_L}(c_n^2)$ spanning $H^0(M(u_L),\lambda_{u_L}(c_n^2))$.  The way we find $\mg_i$ is somehow tricky.

The structure of the paper is arranged as follows.  After collecting notations and preliminaries in the next subsection, in \S 2 we will prove some useful properties related to the moduli space $M(c_n^r)$,  which may overlap some of other's work before.  In \S 3 we find the required divisor $\ts_r$ of $\lambda_{c_r^r}(u_{K_X^{-1}})$.  In \S 4 we state and prove the first part of our result.  Finally, the last section \S 5 is quite independent from the other 3 previous sections, where we state and prove the second part of our result.  

\subsection{Notations \& Preliminaries.}\label{NaP}
\begin{enumerate}
\item $X$ is a rational surface over the complex number $\mathbb{C}$, with $K_X$ the canonical divisor and $H$ the polarization.
\emph{Assume moreover $-K_X$ is effective.}
If $X=\p^2$, then $H$ is the hyperplane class.  
\item We use the same letter to denote both the line bundles and the corresponding divisor classes, but we write $L_1\otimes L_2$, $L^{-1}$ for line bundles while $L_1+L_2$, $-L$ for the corresponding divisor classes.  Denote by $L_1.L_2$ the intersection number of $L_1$ and $L_2$.  $L^2:=L.L$. 
\item $K(X)$ is the Grothendieck group of coherent sheaves over $X$.  $\forall~c\in K(X)$, $M^H_X(c)$ is the moduli space of $H$-semistable sheaves of class $c$.
\item Define $u_L:=[\mo_X]-[L^{-1}]+\frac{(L.(L+K_X))}2[\mo_x]\in K(X)$ with $L$ a line bundle over $X$ and $x$ a single point in $X$.  It is easy to check $u_{\mo_X}=0$ and $u_{L_1}+u_{L_2}=u_{L_1\otimes L_2}$.   Let $\ml:=M^H_X(u_L)$.
\item Define $\crn=r[\mo_{X}]-n[\mo_x]\in K(X)$ with $x$ a single point on $X$.  Let $\wrn:=M^H_X(c_n^r)$.
\item Denote by $\lcl$ the determinant line bundle associated to $u_L$ over (an open subset of) $\wrn$, and simply by $\lrl$ if $r=n$.  
Let $\wrn^L$ be the biggest open subset of $\wrn$ where $\lcl$ is well-defined.  Notice that the stable locus $\wrn^s\subset\wrn^L$. 
 
\item Denote by $\lambda_{L}(\crn)$ the determinant line bundle associated to $\crn$ over $\ml$ and simply $\lambda_L(r)$ if $n=r$.  Notice that $\lambda_{L}(\crn)$ is always well-defined over the whole moduli space.  
\item We denote by $\z_L$ the determinant line bundle associated to $c^1_0=[\mo_{X}]$ on $\ml$.  Then $\z_L$ has a divisor $D_{\z_L}$ which consists of sheaves with non trivial global sections.  Moreover by Proposition 2.8 in \cite{LP2},  $\lambda_L(\crn)\cong\z_L^{\otimes r}\otimes\pi^{*}\mo_{\ls}(n)=:\z^r_L(n)$ where $\pi:\ml\ra\ls$ sends each sheaf to its support.
\item Let $\mf$, $\mg$ be two sheaves.  Then
\begin{itemize}
\item Denote by $r(\mf)$, $c_i(\mf)$ and $\chi(\mf)$ the rank, the i-th Chern class and the Euler characteristic of $\mf$ respectively;
\item $h^i(\mf)=dim~H^i(\mf)$ and hence $\chi(\mf)=\sum_{i\geq0}(-1)^i h^i(\mf)$;
\item $\text{ext}^i(\mf,\mg)=dim~\Ext^i(\mf,\mg)$, $\text{hom}(\mf,\mg)=dim~\Hom(\mf,\mg)$ and $\chi(\mf,\mg)=\sum_{i\geq0}(-1)^i\text{ext}^i(\mf,\mg)$;   
\item If $\mf$ is a 1-dimensional sheaf, then $Supp(\mf)$ or $C_{\mf}$ is the (schematic) support of $\mf$;
\item Write $\mf\in M^H_X(c)$ if the $S$-equivalence class of $\mf$ is in $M^H_X(c)$.
\end{itemize}
\item By abuse of notation, except otherwise stated, 
we always denote by $q$ and $p$ the morphisms from $X\times M$ to $X$ and $M$ respectively, where $X$ is the surface and $M$ can be any moduli space, e.g. $\wrn$, $\ml$, etc.. 
\end{enumerate}

\section{Some properties of $\wrn$.}
The moduli space $\wrn$ may depend on the polarization $H$.  We first extend the concept of walls (see e.g. Section 2.2 in \cite{GY}) to rank $r\geq 2$ cases.
\begin{defn}A collection $\underline{\xi}:=\{\xi_1,\cdots,\xi_t\}$ with $t\leq r$, $\xi_i\in H^2(X,\mathbb{Z})$ for $i=1,\cdots,t$ and $\underline{\xi}\neq\{0,\cdots,0\}$ is called \emph{a collection of type $c^r_n$} 
if the following conditions hold
\begin{enumerate}
\item $\displaystyle{\sum_{i=1}^t} \xi_i=0$;
\item $\exists ~r_i\in\bz_{+}$ for $i=1,\cdots,t$, such that $\sum_{i=1}^t r_i=r$ and \begin{equation}\label{lbxi}\sum_{i=1}^t\frac{\xi_i^2}{r_i}+2n-r+\sum_{i=1}^n\frac1{r_i}\geq 0.\end{equation}
\end{enumerate}

\end{defn}

Denote by $\ma$ the ample cone of $X$.  For a collection $\underline{\xi}$ of type $c_n^r$ we define $$\tw^{\underline\xi}:=\big\{x\in\ma~\big|~ x.\xi_i=0,\forall~i=1,\cdots,t\big\}.$$
  Then $\tw^{\underline\xi}$ is called \emph{a wall of type $c_n^r$}.
Since (\ref{lbxi}) provides a lower bound for $\xi^2_i$, $\tw^{\underline\xi}$ are locally finite in $\ma$.  We call a polarization $H$ is $c_n^r$-general if it does not lie on any wall.   
\begin{lemma}
Let $\mf$ be a strictly $H$-semistable sheaf of class $c_n^r$ which is $S$-equivalent to $\displaystyle{\bigoplus_{i=1}^t}\mf_i$ with $\mf_i$ stable.  Then either $c_1(\mf_1)=\cdots=c_1(\mf_t)=0$ or $\underline{\xi}:=\{c_1(\mf_1),\cdots,c_1(\mf_t)\}$ is a collection of type $c_n^r$ and $H\in \tw^{\underline\xi}$.
\end{lemma}
\begin{proof}Assume not all $c_1(\mf_i)$ are zero.  Let $\xi_i:=c_1(\mf_i)$, $r_i:=r(\mf_i)$ and $a_i:=c_2(\mf_i)$.   Then  
\[\chi(\mf_i,\mf_i)=r^2_i+2r_i(\frac{\xi_i^2}2-a_i)-\xi_i^2.\]
Since $\mf_i$ is stable and $-K_X$ is effective, $\chi(\mf_i,\mf_i)=1-\text{ext}^1(\mf_i,\mf_i)\leq 1$.  Hence
\begin{equation}\label{xi1}a_i-\frac{\xi_i^2}2\geq \frac{-1-\xi_i^2}{2r_i}+\frac{r_i}2.\end{equation}
On the other hand $\displaystyle{\sum_{i=1}^tr_i}=r(\mf)=r$, $\displaystyle{\sum_{i=1}^t\xi_i}=c_1(\mf)=0$ and 
\begin{equation}\label{xi2}n=c_2(\mf)=\sum_{i=1}^t a_i+\sum_{i<j}\xi_i.\xi_j=\sum_{i=1}^t(a_i-\frac{\xi_i^2}2).\end{equation}
Combine (\ref{xi1}) and (\ref{xi2}) and we get
\[\sum_{i=1}^t\frac{\xi_i^2}{r_i}+2n-r+\sum_{i=1}^n\frac1{r_i}\geq 0,\]
hence $\underline\xi$ is a collection of type $c_n^r$ and hence the lemma.
\end{proof}

\begin{coro}\label{cngh}If $H$ is $c_n^r$-general, then $\wrn=\wrn^L$ for any line bundle $L$ over $X$.
\end{coro}
\begin{proof}Let $\mf=\displaystyle{\bigoplus_{i=1}^t} \mf_i\in \wrn^L$ where $\mf_i$ are $H$-stable sheaves with the same reduced Hilbert polynomial.  By descent theory, $\mf\in\wrn^L$ if and only if $c_1(\mf_i).L=0$
for $i=1,\cdots,t.$  Since $H$ is $c_n^r$-general, then $c_1(\mf_i)=0,\forall~i$ and hence $\wrn=\wrn^L$.
\end{proof}

\begin{rem}\label{big2}If $\xi_i.H=0$ and $\xi_i\neq 0$, then by Hodge index theorem $\xi_i^2\leq-1$.  
Moreover since $H^0(\mo_X(\pm\xi_i))=H^2(\mo_X(\pm\xi_i))=0$ by stability, we have 
\begin{equation}\label{chxi}0\geq\chi(\mo_X(\pm\xi_i))=1+\frac{(\xi_i.(\xi_i\pm K_X))}2.\end{equation}  Therefore 
\[\xi_i^2+2\leq K_X.\xi_i\leq -2-\xi_i^2,\]
and hence $\xi_i^2\leq -2$. 
\end{rem}
\begin{lemma}\label{sscod}Let $n\geq r$.  Assume moreover $r\geq3$ or $n\geq 3$.  Then $\wrn\setminus \wrn^s$ is of codimension $\geq 3$ in $\wrn$.  In particular for any line bundle $L$, $\wrn\setminus\wrn^L$ is of codimension $\geq 3$.  \end{lemma}
\begin{proof}

Notice that $dim~W(r,0,n)^s=r(2n-r)+1$.  

We first assume $X=\p^2$ or $H$ is $\crn$-general,  then $\wrn\setminus \wrn^s$ is of codimension $\geq 2$ in $\wrn$ by Theorem 6 in \cite{Dr}.  We can sharpen the result by a direct computation as follows.
\begin{eqnarray}\label{gcod}&&dim~\wrn\setminus\wrn^s \\ &=&\max_{\substack{\sum_{i=1}^t r_i =r \\ \sum_{i=1}^t n_i=n\\ r_in=rn_i,~\forall~i}}
\big\{\sum_{i=1}^t dim~W(r_i,0,n_i)^s\big\}=\max_{\substack{\sum_{i=1}^t r_i =r \\ \sum_{i=1}^t n_i=n\\ r_in=rn_i,~\forall~i}}
\big\{\sum_{i=1}^t r_i(2n_i-r_i)+1\big\}\nonumber\\
&=& \max_{\substack{\sum_{i=1}^t r_i =r \\ \sum_{i=1}^t n_i=n\\ r_in=rn_i,~\forall~i}}
\big\{r(2n-r)-(-t+\sum_{i\neq j} r_i(2n_j-r_j))\big\}\nonumber
\end{eqnarray}  
We only need to show that $$-t+\displaystyle{\sum_{i\neq j}} r_i(2n_j-r_j)\geq 2,~ for~ any ~\{t;r_i;n_i\}.$$
Since $n\geq r$, we have $2n_i-r_i\geq n_i\geq r_i$.  If $t\geq3$, then $-t+\sum_{i\neq j} r_i(2n_j-r_j)\geq 2\binom{t}{2}-t=t(t-2)\geq 3.$. If $t=2$, then let $r':=\min\{r_1,r_2\}$ and we have $-2+r'(2(n-n')-(r-r'))+(r-r')(2n'-r')=-2+2\frac{r'(r-r')}{r}(2n-r)$.  If moreover $r\geq 4$, then $-2+2\frac{r'(r-r')}{r}(2n-r)\geq -2+\frac{2(r-1)}r(2n-r)\geq 2$.  If $r\leq3$, then $r'=1$ and $-2+2\frac{(r-1)}{r}(2n-r)=\begin{cases}&-2+(2n-r)\geq 2,~r=2,n\geq3\\&-2+\frac43(2n-r)\geq 2,~r=3\end{cases}$.  Hence we are done for this case.

We now assume $H$ is not $\crn$-general.

Since $\tw^{\underline\xi}$ are locally finite in $\ma$.  Any polarization $H$ lies on at most finitely many walls.  To prove the lemma it is enough to show the set
\[\wrn\supset\wrn^{\underline\xi}:=\big\{\bigoplus_{i=1}^t\mf_i\big| \mf_i~are~stable~and~c_1(\mf_i)=\xi_i\big\}\]
is of dimension $\leq dim~\wrn^s-3=r(2n-r)-2$.

Denote by $[\mf_i]$ the class of $\mf_i$ in $K(X)$.  Let $r_i=r(\mf_i)$ and $a_i=c_2(\mf_i)$.  
We first assume $\xi_i\neq0$ for all $i$.  By Remark \ref{big2}, $\xi_i^2\leq -2$ if $\xi_i\neq 0$ and by (\ref{xi1}) we have $a_i-\frac{\xi_i^2}2\geq0$.  By (\ref{xi2}) $\displaystyle{\sum_{i=1}^t}(a_i-\frac{\xi_i^2}2)=n$.
Hence
\begin{eqnarray}&&\qquad\quad~dim~\wrn^{\underline\xi}=\sum_{i=1}^t dim ~M^H_X([\mf_i])^s
\\&=&\sum_{i=1}^t (1+\xi_i^2-r_i^2+2r_i(a_i-\frac{\xi_i^2}2))\nonumber\\
&=& 2(\sum_{i=1}^t (a_i-\frac{\xi_i^2}2))(\sum_{i=1}^t r_i)-(\sum_{i=1}^tr_i)^2+\sum_{i\neq j}(r_ir_j-2(a_i-\frac{\xi_i^2}2)r_j)+\sum_{i=1}^t (1+\xi_i^2)\nonumber
\\&=& 2rn-r^2+\sum_{j=1}^t r_j(\sum_{i\neq j}(r_i-2(a_i-\frac{\xi_i^2}2)))+\sum_{i=1}^t(1+\xi_i^2).\nonumber
\end{eqnarray}
By (\ref{xi1}), $r_i-2(a_i-\frac{\xi_i^2}2)\leq \frac{1+\xi_i^2}{r_i}$.  Hence
\begin{eqnarray}dim~\wrn^{\underline\xi}&\leq& 2rn-r^2+\sum_{i=1}^t \frac{1+\xi^2_i}{r_i}(\sum_{j\neq i}r_j)+\sum_{i=1}^t(1+\xi_i^2)\nonumber\\
&=& 2rn-r^2+\sum_{i=1}^t  \frac{1+\xi^2_i}{r_i}(r-r_i)+\sum_{i=1}^t(1+\xi_i^2)\nonumber\\ &=&2rn-r^2+r\sum_{i=1}^t\frac{1+\xi_i^2}{r_i}
\leq 2rn-r^2-t-(\sum_{i<j}(\frac{r_i}{r_j}+\frac{r_i}{r_j}))\nonumber\\
&\leq& 2rn-r^2-t^2\leq 2rn-r^2-4.
\end{eqnarray}

If $\xi_1=\cdots=\xi_s=0$ and $\xi_i\neq0$ for all $s+1\leq i\leq t$, then $\displaystyle{\bigoplus_{i=s+1}^t}\mf_i\in W(r',0,n')^{\underline\xi'}\subset W(r',0,n')$ where $r'=\displaystyle{\sum_{i=s+1}^t}r_i$,
$n'=\displaystyle{\sum_{i=s+1}^t}(a_i-\frac{\xi_i^2}2)$ and $\underline\xi'=\{\xi_{s+1},\cdots,\xi_t\}$ is a collection of type $c_{n'}^{r'}$.  Then
by previous argument we have $dim~ W(r',0,n')^{\underline\xi'}\leq dim~ W(r',0,n')$ and (\ref{gcod}) applies.  

The lemma is proved.
\end{proof}
\begin{coro}\label{dense}$\wrn^s$ is dense in $\wrn$.
\end{coro}
\begin{rem}\label{irred}It is well known that $W(r,L,n)$ is irreducible (Theorem D in \cite{DLP} for $X=\p^2$, Theorem 1 in \cite{Wal} for other rational surfaces, both based on the method of \cite{Ell}).  $W(r,L,n)$ is normal and Cohen-Macaulay because it is a quotient of a smooth variety.  
\end{rem}

We now study the $\mu$-semistable sheaves and have the following useful lemma.
\begin{lemma}\label{nogs}For any $\mu$-semistable (w.r.t. $H$) sheaf $\mf$ with $r(\mf)>0$ and $c_1(\mf).H=0$, we must have either $\chi(\mf)<r(\mf)$ or $\mf\cong\mo_X^{\oplus r(\mf)}$.

Moreover if $\mf$ is an $H$-semistable sheaves with $r(\mf)>0$ and $c_1(\mf).H=0$, then $H^0(\mf)\neq 0\Leftrightarrow \mf\cong\mo_X^{\oplus r(\mf)}$; in particular if $\mf\not\cong\mo_X^{\oplus r(\mf)}$, then $\chi(\mf)\leq 0$.  Therefore $\wrn$ is empty if $n<r$ and $n\neq0$, and $W(r,0,0)=\{\mo_X^{\oplus r}\}$.
\end{lemma}
\begin{proof}We do induction on the rank $r(\mf)$.  
If $r(\mf)=1$, $\mf$ is of form $\mi_Z(\xi)$ with $Z$ a 0-dimensional subscheme of $X$.  $\chi(\mf)=\chi(\mo_X(\xi))-len(Z)$ and by (\ref{chxi}) $\chi(\mo_X(\xi))\leq 0$ unless $\xi=0$.  Hence either $\chi(\mf)\leq0$ or $\mf\cong\mo_X$.

Let $r(\mf)\geq 2$.  If $\chi(\mf)\leq 0$, then we are done.  Now assume $\chi(\mf)>0$.  Since by $\mu$-semistability $H^2(\mf)=\Hom(\mf,K_X)^{\vee}=0$,  $h^0(\mf)\geq \chi(\mf)>0$.
We then have the following exact sequence
\[0\ra\mo_X\ra\mf\ra\mf'\ra0,\]
where $\mf'$ is $\mu$-semistable of rank $r(\mf)-1$ and $\chi(\mf')+1=\chi(\mf)$.  By induction assumption, either $\chi(\mf')<r(\mf')=r(\mf)-1$ or $\mf'\cong\mo_X^{\oplus r(\mf)-1}$.  Therefore either $\chi(\mf)<r(\mf)$ or $\mf\cong \mo_X^{\oplus r(\mf)}$. 

If $\mf$ is $H$-semistable with $r(\mf)>0$, $c_1(\mf).H=0$ and $H^0(\mf)\neq 0$, then by stability $\frac{\chi(\mf)}{r(\mf)}\geq \frac{\chi(\mo_X)}{r(\mo_X)}=1$.  But on the other hand $\mf$ is $\mu$-semistable with $\chi(\mf)\geq r(\mf)$ and hence $\mf\cong\mo_X^{\oplus r(\mf)}$.  We have proved the lemma.
\end{proof}
\begin{rem}\label{mfdl}Let $\mf$ be $\mu$-semistable.  Let $\mf^{\vee}:=\mathcal{H}om(\mf,\mo_X)$.  Then $\mf^{\vee}$ is $\mu$-semistable.  Let $c_1(\mf)=\xi\neq0$ and $\xi.H=0$.  Then by Lemma \ref{nogs},
$\chi(\mf)<r(\mf)$ and also $\chi(\mf^{\vee})<r(\mf^{\vee})=r(\mf)$.  Hence by Riemann-Roch
\begin{equation}\label{kxi}-c_2(\mf)+\frac{\xi^2}2-\frac{K_X.\xi}2=\chi(\mf)-r(\mf)<0;\end{equation}
\begin{equation}\label{kxid}-c_2(\mf)+\frac{\xi^2}2+\frac{K_X.\xi}2=\chi(\mf^{\vee})-r(\mf)-\chi(\me xt^1(\mf,\mo_X))<0.\end{equation}
Notice that $\me xt^1(\mf,\mo_X)$ is a 0-dimensional sheaf.
\end{rem}
Let $\kw(r,L,n)^{\mu}$ be the stack (only a stack in general) of $\mu$-semistable sheaves of rank $r$, determinant $L$ and second Chern class $n$.  Then $\kw(r,L,n)^{\mu}$ is smooth of dimension $r(2n-r)-(r-1)L^2$, since $\Ext^2(\mf,\mf)=0$ for any $\mu$-semistable sheaf $\mf$.  Let $\kw(r,L,n)^{\mu s}$ ($\kw(r,L,n)$, $\kw(r,L,n)^s$, resp.) be the substack of $\kw(r,L,n)^{\mu}$ parametrizing $\mu$-stable ($H$-semistable, $H$-stable, resp.) sheaves.  Then we have
\[\kw(r,L,n)^{\mu s}\subset \kw(r,L,n)^s\subset \kw(r,L,n)\subset\kw(r,L,n)^{\mu}.\]
Moreover we have the following lemma.
\begin{lemma}\label{mustable}
$dim~(\kwrn^{\mu}\setminus\kwrn^{\mu s})\leq \begin{cases}r(2n-r)-2, &\text{if }n>r \\r(2n-r)-1, &\text{if } n=r\end{cases}$
\end{lemma} 
\begin{proof}If $r=1$, there is nothing to prove.  Let $r\geq2$ and let $\mf\in \kwrn^{\mu}\setminus\kwrn^{\mu s}$, then there is an exact sequence
\begin{equation}\label{mufil}0\ra\mg_1\xrightarrow{j_1} \mf\xrightarrow{j_2} \mg_2\ra 0,\end{equation}
where $\mg_i\in\kw(r_i,\xi_i,a_i)^{\mu}$ with $\xi_i.H=0$, and moreover $\mg_2\in\kw(r_2,\xi_2,a_2)^{\mu s}$.  Hence $\Ext^2(\mg_2,\mg_1)\cong \Hom(\mg_1,\mg_2(K_X))^{\vee}=0$.  
Also by direct computation we have 
\begin{equation}\label{rexi}\begin{cases} \xi_1+\xi_2=0; \\  a_1+a_2+\xi_1.\xi_2=n;\\
r_1+r_2=r.\end{cases}
 \end{equation} 

Fix $\Xi:=(r_i,\xi_i,a_i)$, and let $\ke(\Xi)$ be the substack of $\kwrn$ parametrizing sheaves in the middle of (\ref{mufil}).  It is enough to show 
\[dim~\ke(\Xi)\leq \begin{cases}r(2n-r)-2, &\text{if }n>r \\r(2n-r)-1, &\text{if } n=r\end{cases}.\]

Let $Aut(\mf)^o$ be the subgroup of $Aut(\mf)$ containing all the automorphisms $\tau$ of $\mf$ satisfying that $\tau(j_1(\mg_1))\subset j_1(\mg_1)$, which is equivalent to $j_2\circ\tau\circ j_1=0$ and also equivalent to that $\tau$ induces an element $(\tau_1,\tau_2)\in Aut(\mg_1)\times Aut(\mg_2)$.  Then we have a map $Aut(\mf)^o\ra Aut(\mg_1)\times Aut(\mg_2)$ with kernel isomorphic to $\Hom(\mg_2,\mg_1)$.  On the other hand $Aut(\mg_1)\times Aut(\mg_2)$ acts on $\Ext^1(\mg_2,\mg_1)$ and the stabilizer of the extension in (\ref{mufil}) is isomorphic to $Aut^o(\mf)/\Hom(\mg_2,\mg_1)$.

Let $\kp(\Xi)$ be the stack parametrizing the quotient $[\mf\twoheadrightarrow \mg]$ such that $\mf\in\kwrn^{\mu}$ and $\mg\in\kw(r_2,\xi_2,a_2)^{\mu s}$.  $Aut([\mf\twoheadrightarrow \mg])=Aut(\mf)^o$.  Then we have a surjective map $\kp(\Xi)\ra \ke(\Xi)$ and the dimension of the fiber over $\mf$ is at least $dim~Aut(\mf)/Aut(\mf)^o\geq 0$.  Hence   
$dim~\ke(\Xi)\leq dim~\kp(\Xi)$.

On the other hand we have a map $\kp(\Xi)\ra \kw(r_1,\xi_1,a_1)^{\mu}\times \kw(r_2,\xi_2,a_2)^{\mu s}$ whose fiber over $(\mg_1,\mg_2)$ is the stack $\mext^1(\mg_2,\mg_1)$ associated to $\Ext^1(\mg_2,\mg_1)$.  For every extension $\theta\in\Ext^1(\mg_2,\mg_1)$, $\Hom(\mg_2,\mg_1)\subset Aut(\theta)$.  Hence by $\Ext^2(\mg_2,\mg_1)=0$,
$dim~\mext^1(\mg_2,\mg_1)\leq -\chi(\mg_2,\mg_1)=-r_2(\frac{\xi_1^2}2-a_1-\frac{K_X.\xi_1}2)-r_1(\frac{\xi_2^2}2-a_2+\frac{K_X.\xi_2}2)-r_1r_2+\xi_1.\xi_2=:-\chi(\Xi)$.  Therefore by (\ref{rexi}) we have

\begin{eqnarray}&&dim~\ke(\Xi)\leq dim~\kp(\Xi)\leq dim~\kw(r_1,\xi_1,a_1)^{\mu }+dim~ \kw(r_2,\xi_2,a_2)^{\mu s}-\chi(\Xi)\nonumber\\
&=&\sum_{i=1}^2 (r_i(2a_i-r_i)-(r_i-1)\xi_i^2)-\chi(\Xi)\nonumber\\
&=&r(2n-r)+2r_1(-a_2+\frac{\xi_2^2}2)+2r_2(-a_1+\frac{\xi_1^2}2)+2\xi_1^2-\chi(\Xi)\nonumber\\
&=&r(2n-r)+r_1(-a_2+\frac{\xi_2^2}2-\frac{K_X.\xi_2}2)+r_2(-a_1+\frac{\xi_1^2}2+\frac{K_X.\xi_1}2)+r_1r_2+\xi_1^2\nonumber\end{eqnarray}

$-a_i+\frac{\xi_i^2}2\pm\frac{K_X.\xi_i}2<0$ for $i=1,2$ and $-a_2+\frac{\xi_2^2}2-\frac{K_X.\xi_2}2+r_2\leq0$ or $\mg_2\cong \mo_X$ by Lemma \ref{nogs} and Remark \ref{mfdl}.  
Also either $\xi_1^2\leq -2$ or $\xi_1=0$ by Remark \ref{big2}.  Hence if $\xi_1\neq 0$, then $dim~\ke(\Xi)\leq dim~\kwrn-3$.  If $\xi_1=0$ and $\mg_2\not\cong\mo_X$, then $dim~\ke(\Xi)\leq dim~\kwrn-\min\{3,(n-1)\}$.  If $\mg_2\cong\mo_X$, then $r_2=1$, $\xi_1=\xi_2=0$ and $a_2=0$, $a_1=n$.  Therefore
 \[dim~\ke(\mg_2\cong\mo_X)=r(2n-r)-n+r-1\leq\begin{cases}r(2n-r)-2, &\text{if }n>r \\r(2n-r)-1, &\text{if } n=r\end{cases}.\]

We have proved the lemma.
\end{proof}

\begin{coro}\label{gevan}Let $r\geq 2$ and $n\neq0$.  For a generic sheaf $\mf\in \wrn$, $\Hom(\mf,\mo_X)=0$.
\end{coro}

\begin{lemma}\label{lowrk}Assume we have a non-splitting sequence as follows
\begin{equation}\label{oext}0\ra \mf_n^{r-1}\ra \mf_n^r\ra \mo_X\ra 0,
\end{equation}
where $\mf_n^{r-1}$ ($\mf_n^r$, resp.) is a sheaf of class $c_n^{r-1}$ ($\crn$, resp.).  Then $\mf_n^r$ is $\mu$-semistable iff $\mf_n^{r-1}$ is.  

In particular let $n=r$, then $\mf_r^r$ is semistable if $\mf_r^{r-1}$ is semistable (hence stable since $g.c.d.(r,r-1)=1$); and conversely $\mf_r^{r-1}$ is stable if $\mf_r^r$ is stable.  

\end{lemma}
\begin{proof}Because $c_1(\mf_n^{r-1})=c_1(\mf_n^r)=c_1(\mo_X)=0$ and $\mo_X$ is $\mu$-semistable, $\mf_n^{r}$ is $\mu$-semistable iff $\mf_n^{r-1}$ is.  


Let $r=n$.  Assume $\mf_r^{r-1}$ is semistable.  Then $\mf_r^r$ is $\mu$-semistable.  Let $\mg\subset \mf_r^r$ with $c_1(\mg).H=0$.  It suffices to show $\chi(\mg)\leq0$.  Let $\mg_1=\mg\cap \mf_r^{r-1}$ and $\mg/\mg_1$ is a subsheaf of $\mo_X$.  Then $c_1(\mg_1).H=c_1(\mg/\mg_1).H=0$ because both $\mo_X$ and $\mf_r^{r-1}$ are $\mu$-semistable.  By stability of $\mf_r^{r-1}$ either $\chi(\mg_1)<0$ or $\mg_1=0$.   Also either $\chi(\mg/\mg_1)\leq 0$ or $\mg/\mg_1\cong\mo_X$.  Therefore either $\chi(\mg)\leq 0$ or $\mg\cong \mo_X$.  But (\ref{oext}) does not split.  So $\chi(\mg)\leq 0$ and $\mf_r^r$ is semistable.

Now assume $\mf_r^r$ is stable, then $\mf_{r}^{r-1}$ is $\mu$-semistable.  Let $\mg'\subset\mf_r^{r-1}\subset\mf_r^r$ with $c_1(\mg').H=0$.  It suffices to show that $\frac{\chi(\mg')}{r(\mg')}<-\frac1{r-1}=\frac{\chi(\mf_r^{r-1})}{r(\mf_r^{r-1})}$.  $\chi(\mg')\leq -1$ by stability of $\mf_r^r$, hence $\frac{\chi(\mg)}{r(\mg)}<-\frac1{r-1}=\frac{\chi(\mf_r^{r-1})}{r(\mf_r^{r-1})}$ if $r(\mg')<r-1$.  If $r(\mg')=r-1$, then $c_1(\mf_r^{r-1}/\mg').H=0=r(\mf_r^{r-1}/\mg')$ and hence $\mf_r^{r-1}/\mg'$ is a 0-dimensional sheaf and $\frac{\chi(\mg')}{r(\mg')}<\frac{\chi(\mf_r^{r-1})}{r(\mf_r^{r-1})}$.  

The lemma is proved.
\end{proof}
\begin{rem}\label{new}Let $\mf_r^{r-1}$ and $\mf_r^{r}$ be as in Lemma \ref{lowrk}.  Then 
\[\mf_r^{r-1}~ is~ stable~ and~ locally~ free ~\Leftrightarrow~\mf_r^r~ is~ stable~ and~ locally~ free.\]  This is because when $\mf_r^{r-1}$ is stable, $\mf_r^r$ is strictly semistable only if (\ref{oext}) splits along the ideal sheaf $\mi_x\subset\mo_X$ of some single point $x\in X$, which is not possible if $\mf_r^{r-1}$ is locally free.
\end{rem}
\begin{flushleft}{\textbf{Convension.}}
From now on, we will deal with global sections of determinant line bundles $\lcl$ over the moduli space $\wrn^L$.  By Lemma \ref{sscod}, $\wrn\setminus\wrn^L$ is always of codimenison $\geq 3$.  Hence without loss of generality, we may assume the polarization $H$ is general enough so that we can always write $\wrn$ instead of $\wrn^L$ for any $L$.
\end{flushleft}

\section{A closed subscheme $\ts_r$ of $\wrr$.}
\begin{prop}\label{ssos}Let $r\geq 2$.
There is a canonical section $s_r$ of the determinant line bundle $\lambda_r(K_X^{-1})$ over $\wrr$ whose zero set is 
\begin{equation}\label{ds}\ts_r:=
\big\{\mf\in \wrr | \Hom(\mf,\mo_X)\neq 0\big\}.\end{equation}
Moreover $\ts_r$ is reduced as a divisor associated to $\lambda_r(-K_X)$ of $\wrr$, and there is a birational morphism $\delta:W(r-1,0,r)\ra \ts_r$ which is surjective on the stable locus.  Moreover for any line bundle $L$ over $X$, $\delta^{*}\lambda_r(L)\cong \lambda_{c_r^{r-1}}(L).$
\end{prop}
\begin{proof}Since $-K_X$ is effective, we have a curve $C\in |-K_X|$ and its structure sheaf $\mo_C$ is of class $u_{K_X^{-1}}$.  Hence by Lemma 2.3 and Lemma 2.4 in \cite{Da2}, there is a section $s_{\mo_C}$ (unique up to scalars) of line bundle $\lambda_r(K_X^{-1})$ over $W(r,0,r)$ whose zero set is 
$$D_{\mo_C}:=\big\{\mf\in W(r,0,r) | H^1(\mf\otimes\mo_C)\neq 0\big\}.$$

On the other hand we have the following exact sequence
\[0\ra K_X\ra \mo_X\ra\mo_C\ra 0.\]
We then have the following exact sequence
\begin{equation}\label{kcurve}H^1(\mf)\ra H^1(\mf\otimes\mo_C)\ra H^2(\mf(K_X))\ra H^2(\mf).\end{equation}
Since $\mf\in\wrr$, $H^0(\mf)=H^2(\mf)=0$ by semistability of $\mf$.  Also $\chi(\mf)=0$, therefore $H^1(\mf)=0$.  Hence (\ref{kcurve}) gives
\[H^1(\mf\otimes\mo_C)\cong H^2(\mf(K_X))\cong \Hom(\mf,\mo_X)^{\vee}.\]
Hence $D_{\mo_C}=\ts_r$ as sets.  

We denote also by $\ts_r$ the reduced subscheme, it is enough to show 
\begin{equation}\label{div}\mo_{W}(\ts_r)\cong \lambda_r(K_X^{-1}).\end{equation}  
If $r=2$, then we are done by Corollary 2.8 in \cite{Yuan6}.  
We assume $r\geq 3$.   By Lemma \ref{lowrk}, we have a morphism $\delta:W(r-1,0,r)\ra \ts_r$ which is surjective on the stable locus.  $\ts_r$ is Cohen-Macaulay since so is $\wrr$.

Recall that $\kw(r,0,r)$ and $\kw(r-1,0,r)$ are the stacks associated to $W(r,0,r)$ and $W(r-1,0,r)$ respectively.   $\kw$ and $\kw(r-1,0,r)$ are smooth.  Let $\kz_r$ be the substacks of $\kw(r,0,r)$ associated to $\ts_r$.  We also denote by $\delta:\kw(r-1,0,r)\ra \kz_r$ the same map at stack level.  Then by Lemma \ref{mustable}, $\delta$ is surjective outside of a substack of codimension $\geq2$.

Denote by $\kw^o$ the open substack of $\kw(r,0,r)$ consisting of sheaves $\mf$ such that $\text{hom}(\mf,\mo_X)\leq 1$.  Let $\kz_r^o=\kz_r\cap \kw^o$ and $\kw(r-1,0,r)^o=\delta^{-1}(\kz^o_r)$.  Easy to see that $\delta$ restricted to $\kw(r-1,0,r)^o$ is an isomorphism to its image.  On the other hand by Lemma \ref{mustable} and Lemma \ref{lowrk}, we have $codim~(\kw(r-1,0,r)\setminus \kw(r-1,0,r)^o)\geq 2$ and $codim~(\kz_r\setminus \kz_r^o)\geq 2$.  Denote by $\ts_r^o$ the image of $\kz^o_r$ via the corepresentation $\kz_r\ra \ts_r$.  Then by Lemma \ref{sscod} $codim~(\ts_r\setminus\ts_r^o)\geq 2$ and hence $\ts_r$ is normal.  

Let $\mn_{\kz_r^o/\kw^o}$ be the normal bundle of $\kz_r^o$ inside $\kw^o$.  
To show (\ref{div}), it is enough to show $\delta^{*}\mn_{\kz_r^o/\kw^o}\cong \delta^{*}\lambda_{r}(K_X^{-1})$ over $\kw(r-1,0,r)^o$.

Let $\F_r^{r-1}$ be the universal sheaf over $X\times \kw(r-1,0,r)^o$.  Then $\mr:=R^1p_{*}\F_r^{r-1}\cong \E xt^1_p(\mo_{X\times \kw(r-1,0,r)^o},\F_r^{r-1})$ is a line bundle over $\kw(r-1,0,r)^o$.  We have a universal extension over $X\times \kw(r-1,0,r)^o$ as follows.
\begin{equation}\label{unex}
0\ra\F_r^{r-1}\ra \F_r^r\ra p^{*}\mr\ra 0,
\end{equation}
where $\F_r^r$ is the family of sheaves inducing the identification of $\kw(r-1,0,r)^o$ and $\kz_r^o$.  

By the universal property of the determinant line bundles, we have 
\begin{eqnarray}\label{rr1}\delta^{*}\lambda_{r}(K_X^{-1})&\cong& det^{-1}R^{\bullet}p_{*}(\F_r^r\otimes q^{*}\mo_C)\nonumber\\
&\cong& det^{-1}R^{\bullet}p_{*}(\F_r^{r-1}\otimes q^{*}\mo_C)\otimes det^{-1}R^{\bullet}p_{*}(p^{*}\mr\otimes q^{*}\mo_C)\nonumber\\
&\cong&det^{-1}R^{\bullet}p_{*}(\F_r^{r-1}\otimes q^{*}\mo_C)\otimes \mr^{\otimes-\chi(\mo_C)}\nonumber\\
&\cong& det^{-1}R^{\bullet}p_{*}(\F_r^{r-1}\otimes q^{*}\mo_C)=\lambda_{c_r^{r-1}}(K_X^{-1})\end{eqnarray}

Applying the functor $p_{*}\cdot\hh om(\F_r^{r-1},-)$ to (\ref{unex}) we get the following sequence
\begin{equation}\label{tan1}
0\ra\E xt^1_p(\F_r^{r-1},\F_r^{r-1})\ra \E xt^1_p(\F_{r}^{r-1},\F_r^r)\ra \E xt^1_p(\F_r^{r-1},p^{*}\mr)\ra 0,
\end{equation}
where the left zero is because $\hh om(\F_{r}^{r-1},p^{*}\mr)\otimes k(z)=\Hom(\mf_r^{r-1},\mo_X)=0$ and the right zero is because $\E xt^2_p(\F_r^{r-1},\F_r^{r-1})\otimes k(z)=\Ext^2(\mf_r^{r-1},\mf_r^{r-1})=0$ for every $z\in \kz_r^o$.

Applying the functor $p_{*}\cdot\hh om(-,\F_r^{r})$ to (\ref{unex}) we get the following sequence
\begin{equation}\label{tan2}
0\ra\E xt^1_p(\F_r^{r},\F_r^{r})\xrightarrow{\cong} \E xt^1_p(\F_{r}^{r-1},\F_r^r)\ra 0,
\end{equation}
where the left zero is because $\E xt^1_p(p^{*}\mr,\F_r^r)\otimes k(z)=\Ext^1(\mo_X,\mf_r^{r})=0$ and the right zero is because $\E xt^2_p(p^{*}\mr,\F_r^{r})\otimes k(z)=\Ext^2(\mo_X,\mf_r^{r})=0$ for every $z\in \kz_r^o$.

Combine (\ref{tan1}) and (\ref{tan2}) and we have
\begin{equation}\label{tan3}
0\ra\E xt^1_p(\F_r^{r-1},\F_r^{r-1})\ra \E xt^1_p(\F_{r}^{r},\F_r^r)\ra \E xt^1_p(\F_r^{r-1},p^{*}\mr)\ra 0.
\end{equation}
Because $\E xt^1_p(\F_r^{r-1},\F_r^{r-1})$ is the tangent bundle over $\kw(r-1,0,r)^o$ and $\E xt^1_p(\F_{r}^{r},\F_r^r)$ is the pullback of the tangent bundle of $\kw^o$, we have 
\begin{equation}\label{rr2}\delta^{*}\mn_{\kz_r^o/\kw^o}\cong \E xt^1_p(\F_r^{r-1},p^{*}\mr)\cong
det^{-1}R^{\bullet}(p_{*}\cdot\hh om)(\F^{r-1}_r, p^{*}\mr),\end{equation}
where the last isomorphism is because $p_{*}\cdot\hh om(\F^{r-1}_r,\mo_{X\times\kw(r-1,0,r)^o}))=0=\E xt^2_p(\F^{r-1}_r,\mo_{X\times\kw(r-1,0,r)^o}))$.  On the other hand $\F_r^{r-1}$ admits a locally free resolution of finite length, hence by Lemma 5.5 in \cite{Abe} we have
\begin{eqnarray}\label{nors}&&det^{-1}R^{\bullet}(p_{*}\cdot\hh om)(\F^{r-1}_r, p^{*}\mr)\nonumber\\&\cong&
det^{-1}R^{\bullet}(p_{*}\cdot\hh om)(\F^{r-1}_r, \mo_{X\times\kw(r-1,0,r)^o})\otimes\mr\nonumber\\&\cong& det^{-1}R^{\bullet}p_{*}\cdot R^{\bullet}\hh om(\F^{r-1}_r, \mo_{X\times\kw(r-1,0,r)^o})\otimes \mr\nonumber\\
&\cong& (det R^{\bullet}p_{*}\F^{r-1}_r\otimes q^{*}K_X)\otimes \mr\nonumber\\
&\cong&(det R^{\bullet}p_{*}\F^{r-1}_r\otimes q^{*}K_X)\otimes(det^{-1} R^{\bullet}p_{*}\F^{r-1}_r)=\lambda_{c_r^{r-1}}(K_X^{-1}).\end{eqnarray}
Notice that $\mr\cong det^{-1} R^{\bullet}p_{*}\F^{r-1}_r$ because $p_{*}\F^{r-1}_r=R^{2}p_{*}\F^{r-1}_r=0$.

Combining (\ref{rr1}), (\ref{rr2}) and (\ref{nors}), we get $\delta^{*}\mn_{\kz_r^o/\kw^o}\cong \delta^{*}\lambda_{r}(K_X^{-1})$.

For any line bundle $L$ over $X$, by (\ref{unex}) we have the following equation analogous to (\ref{rr1})
\begin{equation}\label{trans}\delta^{*}\lambda_r(L)\cong \lambda_{c_r^{r-1}}(L)\otimes \mr^{-\chi(u_L)}\cong \lambda_{c_r^{r-1}}(L),
\end{equation}
where the last isomorphism is because $\chi(u_L)=0$.  
Hence the proposition.
\end{proof}
\begin{rem}\label{bir}For $r\geq3$, $\ts_r$ is normal and integral.  Moreover $\delta_{*}\mo_{W(r-1,0,r)}\cong \mo_{\ts_r}$ and hence 
$\lambda_r(L)\cong \delta_{*}\delta^{*}\lambda_r(L)\cong\delta_{*}\lambda_{c^{r-1}_r}(L)$. 
Therefore together with Lemma 3.3 in \cite{Yuan6} we have the following isomorphism for $r\geq2$
\begin{equation}\label{ws1}\delta^{*}:~H^0(\ts_r,\lambda_r(L))\xrightarrow{\cong}H^0(W(r-1,0,r),\lambda_{c^{r-1}_r}(L)).\end{equation}
\end{rem}

\section{Higher rank cases.}
\subsection{General results.}\qquad

Let $L$ be a nontrivially effective line bundle.  Let $SD_{c_n^r,L}$ be \emph{the strange duality map} (see e.g. \S 2.2 in \cite{Yuan6}) as follows.
\[SD_{c_n^r,L}:H^0(\wrn,\lcl)^{\vee}\ra H^0(\ml,\z_L^r(n)).\]
We also write $SD_{r,L}:=SD_{c_r^r,L}$.  
Let $SD_{L,c_n^r}$ be the strange duality map dual to $SD_{c_n^r,L}$ as follows.
\[SD_{L,c_n^r}:H^0(\ml,\z_L^r(n))^{\vee}\ra H^0(\wrn,\lcl).\]
We also write $SD_{L,r}:=SD_{L,c_r^r}$.  

Recall that $\z_L\cong \lambda_L([\mo_X])\cong\lambda_L(c_0^1)$ and $\z_L^r(n):=\z_L^{\otimes r}\otimes\pi^{*}\mo_{\ls}(n)\cong\lambda_L(c_n^r)$ with $\pi:\ml\ra\ls$ sending every sheaf to its support.  $\z_L$ has a canonical divisor 
\[D_{\z_L}:=\big\{\mf\in\ml\big| H^0(\mf)\neq0\big\}.\]

On $\ml$ and $\wrr$ we have the following two exact sequences respectively
\begin{equation}\label{zes}0\ra\z_L^{r-1}(r)\ra\z_L^r(r)\ra\z_L^r(r)|_{D_{\z_L}}\ra 0,\end{equation}
\begin{equation}\label{ses}0\ra \lambda_r(L\otimes K_X)\ra\lambda_r(L)\ra\lambda_r(L)|_{\ts_r}\ra0,\end{equation}
where (\ref{ses}) is because of Proposition \ref{ssos}.  

We have the following proposition which generalizes Lemma 3.1 in \cite{Yuan6}.
\begin{prop}\label{hrmain}
Let $r\geq 2$.  By taking the global sections of (\ref{zes}) and the dual of global sections of (\ref{ses}), we have the following commutative diagram
\begin{equation}\label{rsddim}\xymatrix{&H^0(\ts_r,\lambda_r(L)|_{\ts_r})^{\vee}\ar[r]^{~\quad~~g_r^{\vee}}\ar[d]_{\alpha_{\ts_r}}& H^0(\lambda_r(L))^{\vee}\ar[r]^{f_r^{\vee}~~~~}
\ar[d]^{SD_{r,L}}& H^0(\lambda_r(L\otimes K_X))^{\vee}\ar[r]\ar[d]^{\beta_D} &0\\
0\ar[r]& H^0(\z_L^{r-1}(r))\ar[r]_{f_L} & H^0(\z^r_L(r))\ar[r]_{g_L\quad} &H^0(D_{\z_L},\z^r(r)|_{D_{\z_L}})}.
\end{equation}
Moreover 
\begin{equation}\label{ad1}\alpha_{\ts_r}\circ\delta^{*\vee}=SD_{c_r^{r-1},L}:H^0(W(r-1,0,r),\lambda_{c_r^{r-1}}(L))^{\vee}\ra H^0(\ml,\z_L^{r-1}(r)),\end{equation}
where $\delta^{*\vee}$ is the dual to the isomorphism $\delta^{*}$ in (\ref{ws1}).
\end{prop}
\begin{proof}The proof of (\ref{rsddim}) is analogous to \cite{Yuan6}.  We only need to show that $g_L\circ SD_{r,L}\circ g_r^{\vee}=0$ which can be deduced from that $H^0(\mf\otimes\mg)\neq0$ for all $\mg\in D_{\z_L}$ and $\mf\in\ts_r$.  Any $\mf\in\ts_r$ lies in the following sequence
\begin{equation}\label{five}0\ra\mf'\ra\mf\ra\mo_X\ra0.\end{equation}
Tensor (\ref{five}) by $\mg\in D_{\z_L}$, take the global sections and we get 
\[0\ra H^0(\mf'\otimes\mg)\ra H^0(\mf\otimes \mg)\ra H^0(\mg)\ra H^1(\mf'\otimes\mg).\]
If $H^0(\mf'\otimes\mg)=0$, then $H^1(\mf'\otimes\mg)=0$ and hence $H^0(\mf\otimes \mg)\cong H^0(\mg)\neq 0$.  So we are done.

To show (\ref{ad1}), it is enough to show 
\begin{equation}\label{ad2}SD_{r,L}\circ g_{r}^{\vee}\circ\delta^{*\vee}=SD_{c_r^{r-1},L}\circ f_L:H^0(W(r-1,0,r),\lambda_{c_r^{r-1}}(L))^{\vee}\ra H^0(\ml,\z_L^{r}(r)).\end{equation}

Let $\km(L,0)$ be the moduli stack associated to $\ml$.  Then denote by $\G_L$ a universal sheaf over $X\times\km(L,0)$.  Over $X\times\kw(r-1,0,r)^o$ we have a sequence as in (\ref{unex}).  Moreover we can have the following commutative diagram 
\begin{equation}\label{res1}\xymatrix@R=0.4cm@C=0.6cm{&0&0&0&\\0\ar[r]&\F_r^{r-1}\ar[u]\ar[r]&\F_r^r\ar[r]\ar[u]& p^*\mr\ar[r]\ar[u]&0\\
0\ar[r]&\B'\ar[u]\ar[r]&\B\ar[u]\ar[r]&p^{*}\mr\ar[r]\ar[u]^{=}&0\\
&\A\ar[u]\ar[r]^{=}&\A\ar[u]&&\\&0\ar[u]&0\ar[u]&&,}\end{equation}
where $\B$, $\B'$ and $\A$ are locally free and $\B\cong q^{*}\mo_X(-mH)^{\otimes V}$ with $m\gg0$.  
Then over $X\times\km(L,0)\times\kw(r-1,0,r)^o$ we have 
\begin{equation}\label{res2}\xymatrix@R=0.6cm@C=0.6cm{&0&0&0&\\0\ar[r]&\G_L\boxtimes \F_r^{r-1}\ar[u]\ar[r]&\G_L\boxtimes \F_r^r\ar[r]\ar[u]& \G_L\boxtimes p^*\mr\ar[r]\ar[u]&0\\
0\ar[r]&\G_L\boxtimes \B'\ar[u]\ar[r]&\G_L\boxtimes \B\ar[u]\ar[r]& \G_L\boxtimes p^*\mr\ar[r]\ar[u]^{=}&0\\
& \G_L\boxtimes \A\ar[u]\ar[r]^{=}& \G_L\boxtimes \A\ar[u]&&\\&0\ar[u]&0\ar[u]&&}.\end{equation}
Since $\B\cong q^{*}\mo_X(-mH)^{\otimes V}$ with $m\gg0$, we can ask $p_{*}(\G_L\boxtimes \B)=0$.  Also $R^2p_{*}(\G_L\boxtimes\E)=0$ for any coherent $\E$ since the relative dimension of $Supp(\G_L)$ via $p$ is 1.  Hence we have the following commutative diagram (we have two different maps which are both denoted by $p$, by abuse of notation)
\begin{equation}\label{res3}{\footnotesize\xymatrix@R=0.6cm@C=0.6cm{0\ar[r]&p_{*}(\G_L\boxtimes p^*\mr)\ar[r]&R^1p_{*}(\G_L\boxtimes \F_r^{r-1})\ar[r]&R^1p_{*}(\G_L\boxtimes \F_r^r)\ar[r]&R^1p_{*}(\G_L\boxtimes p^*\mr)\ar[r]&0\\ 0\ar[r]&p_{*}(\G_L\boxtimes p^*\mr)\ar[r]\ar[u]^{=}&R^1p_{*}(\G_L\boxtimes \B')\ar[u]\ar[r]^{\eta_{\theta}}&R^1p_{*}(\G_L\boxtimes \B)\ar[r]\ar[u]&R^1p_{*}(\G_L\boxtimes p^*\mr)\ar[r]\ar[u]_{=}&0  \\ 
& &R^1p_{*}(\G_L\boxtimes \A)\ar[u]^{\eta^{r-1}_r}\ar[r]^{=}& R^1p_{*}(\G_L\boxtimes \A)\ar[u]_{\eta_r^r} &&\\& &p_{*}(\G_L\boxtimes \F_{r}^{r-1})\ar[u]\ar[r]& p_{*}(\G_L\boxtimes \F_r^r)\ar[u] &&.
}}\end{equation}
The section $det(\eta_r^{r-1})$ induces the map $SD_{c^{r-1}_r,L}$, while the section $det(\eta_{r}^r)$ induces the map $SD_{r,L}\circ g_r^{\vee}\circ\delta^{*\vee}$.  Also $det(\eta_{\theta})$ is exactly the section associated to the divisor $\delta^*(D_{\z_L})$, hence $SD_{c^{r-1}_r,L}\circ f_L$ is induced by $det(\eta_{\theta})\cdot det(\eta_r^{r-1})$. 
By (\ref{res3}) we have $det(\eta_{r}^r)=det(\eta_{\theta})\cdot det(\eta_r^{r-1})$ and hence we have proved (\ref{ad2}).
\end{proof}
On the map $\beta_D$ in (\ref{rsddim}), we have 
\begin{prop}\label{beta1}Let $(X,L)$ be as in Lemma 3.9 or Lemma 3.10 in \cite{Yuan6}, i.e. either $\cb$ or $\cb'$ in \S 3.2 of \cite{Yuan6} is satisfied, then we have an injective map for all $r>0$
\[j_r: H^0(D_{\z_L},\z_L^r(r)|_{D_{\z_L}})\hookrightarrow H^0(M(L\otimes K_X,0),\z^r_{L\otimes K_X}(r)).\]
Moreover, $j_r\circ\beta_D=SD_{r,L\otimes K_X}$.
\end{prop}
\begin{proof}The proof is the same as Lemma 3.9 in \cite{Yuan6} and hence omitted here.
\end{proof}

\subsection{Applications to $X=\p^2$ or $\Sigma_e$.}\qquad

We won't restate $\cb$ or $\cb'$ in this paper, but according to Theorem 3.14 in \cite{Yuan6} we have 
\begin{coro}\label{beta2}Proposition \ref{beta1} applies to the following cases
\begin{enumerate}
\item $X=\p^2$, $L=dH$ for $d>0$.
\item $X=\p(\mo_{\p^1}\oplus\mo_{\p^1}(e)):=\Sigma_e$ with $F$ the fiber class and $G$ the section such that $G.G=-e$, and 
$L=aG+bF$ is one of the following
\begin{itemize}
\item $\min\{a,b\}\leq 1$;
\item $\min\{a,b\}\geq 2$, $e\neq 1$, $L$ ample;
\item $\min\{a,b\}\geq 2$, $e=1$, $b\geq a+[a/2]$ with $[a/2]$ the integral part of $a/2$.
\end{itemize}
\end{enumerate}
In particular by Proposition \ref{hrmain} for $(X,L)$ as above we have for all $r\geq2$
\[\left.\begin{array}{r}SD_{c_r^{r-1},L}\text{ is injective (surjective, an isomorphism, resp.)}\\
SD_{r,L\otimes K_X}\text{ is injective (surjective, an isomorphism, resp.)}\end{array}\right\}\Rightarrow
\text{So is }SD_{r,L}.\] 
\end{coro}  

In this whole subsection let $(X,L)$ always be as in Corollary \ref{beta2}.
\begin{rem}\label{zmap}Notice that in Corollary \ref{beta2} if $H^0(L\otimes K_X)=0$, then $\z_L=\emptyset$ and hence $\beta_{D}=0$.  In this case, $M(L\otimes K_X,0)=\emptyset$ and we define $SD_{r,L\otimes K_X}=0$.  If $L\otimes K_X\cong\mo_X$, then $M(L\otimes K_X,0)=M(\mo_X,0)$ consists of a single point which is the zero sheaf, and $\lambda_r(\mo_X)\cong\mo_{\wrr}$ over $\wrr$.  We can still define $SD_{r,\mo_X}$ via the locus $\mathscr{D}_{r,\mo_X}$ inside $\wrr\otimes M(\mo_X,0)\cong \wrr$ as follows
\[\mathscr{D}_{r,\mo_X}:=\big\{(\mf,\mg)\big|\mf=0\in M(\mo_X,0),\mg\in\wrr,~and~H^0(\mf\otimes\mg)\neq0\big\}.\]
Since $\mathscr{D}_{r,\mo_X}$ is empty and hence $\lambda_{r}(\mo_X)\boxtimes \lambda_{\mo_X}(r)\cong \mo_{\wrr}$ and $SD_{r,L\otimes K_X}$ is a non-zero map from $H^0(\wrr,\lambda_r(\mo_X))=H^0(\wrr,\mo_{\wrr})\cong\mathbb{C}$ to $\mathbb{C}$, which is an isomorphism. 
\end{rem}
\begin{rem}\label{lack}
As proved in \cite{Yuan6}, $SD_{2,L}$ is an isomorphism.  However, we are still in lack of a suitable bridge from $SD_{r,L}$ to $SD_{c_{r+1}^r,L}$, in order to get the expected properties of $SD_{r,L}$ in general by induction on $r$ and $L$.    
\end{rem}

\begin{lemma}\label{like}If $-K_X-L$ is effective, then $\beta_D$ is an isomorphism.  
Therefore in this case we have
\[SD_{c_r^{r-1},L}\text{ is injective (surjective, an isomorphism, resp.)}\Rightarrow
\text{So is }SD_{r,L}\]
\end{lemma}
\begin{proof}If $L\neq -K_X$, then $D_{\z_L}=\emptyset$ by Proposition 4.1.1 and Corollary 4.3.2 in \cite{Yuan1}.  On the other hand, we want to show that $H^0(\lambda_r(L\otimes K_X))=0$.  It is enough to show that $\lambda_r(L^{-1}\otimes K_X^{-1})$ has a non-zero global section vanishing at some point.  Choose a curve $B\in |-L-K_X|$.  Then we can write $B=\cup_i B_i$ such that $B_i\cong\p^1$.  Let $\mg_B:=\displaystyle{\bigoplus_i}~\mo_{B_i}(-1)$.  There is a natural section $\sigma_B$ of $\lambda_r(L^{-1}\otimes K_X^{-1})$ vanishing at points $\mf\in \wrr$ such that $H^0(\mf\otimes\mg_B)=\displaystyle{\bigoplus_i}~ H^0(\mf\otimes \mo_{B_i}(-1))\neq 0$. 
Let $\mf=\displaystyle{\bigoplus_{j=1}^r} ~\mi_{x_j}$ with $x_j\in X$ and $\mi_{x_j}$ the ideal sheaf of $x_j$.  Then $H^0(\mf\otimes\mg_B)\neq 0\Leftrightarrow \exists~x_j\in B.$. Hence $\sigma_B$ is non-zero and vanishes at some points.  Hence $H^0(\lambda_r(L\otimes K_X))=0$.

If $L=-K_X$, then $D_{\z_L}\cong\ls$ and $\z_L^r(r)|_{D_{\z_L}}\cong \mo_{\ls}$.  Also $\lambda_r(L\otimes K_X)\cong\mo_{\wrr}$.  That $\beta_D$ is an isomorphism follows from Remark \ref{zmap} and Proposition \ref{beta1}.  We also can show $\beta_D$ is an isomorphism directly: it is enough to show $g_L\circ SD_{r,L}$ is not zero.  It is then enough to find a sheaf $\mf\in\wrr$ and a sheaf $\mg\in D_{\z_{-K_X}}$ such that $H^0(\mg\otimes\mf)=0$.  $\mg\in D_{\z_{-K_X}}$ then $\mg\cong\mo_C$ for some $C\in|-K_X|$.  By (\ref{kcurve}) we have $H^0(\mg\otimes\mf)\neq0\Leftrightarrow \mf\in \ts_r$.  Hence $\beta_D$ is an isomorphism. 
We have proved the lemma.
\end{proof}
\begin{rem}\label{S2}Lemma \ref{like} holds not only for $(X,L)$ in Corollary \ref{beta2} but also for all $L$ on $X=\p^2$ or $\Sigma_e$.  For instance, $L=-K_X$ on $X=\Sigma_e$ with $e\geq2$.
\end{rem}

We have $SD_{r,L\otimes K_X}=0$ if $L+ K_X$ is not effective.  However for general $r$, we don't know whether $H^0(\wrr,\lambda_{r}(L\otimes K_X))$ is zero (in other words $\beta_D$ is an isomorphism) or not if neither $L+K_X$ nor $-L-K_X$ is effective.  But we have following proposition for $r=2$ mostly due to Abe (\cite{Abe}).  
\begin{prop}\label{rank2}Let $r=2$.  If $H^0(L\otimes K_X)=0$, then for all $n\geq 2$
\begin{enumerate}\item $H^0(W(2,0,n),\lambda_{c^2_n}(L\otimes K_X))=0$;

\item $SD_{c_{n}^2,L}$ is an isomorphism.
\end{enumerate}
\end{prop}
\begin{proof}If $n=2$, then the proposition follows from Corollary 3.4 and Remark 3.5 in \cite{Yuan6}.  

For $n>2$, Theorem 6.5 in \cite{Abe} and Lemma \ref{mustable} implies (1).  In order to prove (2), it is enough to show that Theorem 7.8 in \cite{Abe} applies, and hence it is enough to check the following four conditions (see \S 7.3 in \cite{Abe}):
\begin{enumerate}
\item[(i)] $K_X.H<0$;
\item[(ii)] $H^1(X,\mo_X)=0$;
\item[(iii)] $\ml$ is irreducible normal and birational to $\ls$;
\item[(iv)] $\kw(2,0,n)$ is irreducible for $n\geq1$. 
\end{enumerate}

By Proposition 4.1.1, $\ml\cong\ls$.  $\kw(2,0,n)$ is irreducible for $n\geq2$ by Lemma \ref{mustable}.  $\mf\in\kw(2,0,1)$ iff $\mf$ lies in the following sequence
\[0\ra\mo_x\ra\mf\ra \mi_x\ra 0,\]
where $\mi_x$ is a ideal sheaf of a single point $x\in X$.  Hence $\mf\cong \mo_X\oplus\mi_x$ since $H^1(\mi_x)=0$.  Therefore $\kw(2,0,1)$ is irreducible.

The proposition is proved.
\end{proof}

By applying Corollary \ref{beta2} and Proposition \ref{rank2}, we get the following theorem. 
\begin{thm}\label{rank3}Let $r=3=n$.  If $H^0(L\otimes K_X)=0$, then
\begin{enumerate}\item $H^0(W(3,0,3),\lambda_{3}(L\otimes K_X))=0$;

\item $SD_{3,L}$ is an isomorphism.
\end{enumerate}
\end{thm}
\begin{proof}Since $-K_X$ is effective, $H^0(L\otimes K_X)=0\Rightarrow H^0(L\otimes K_X^{\otimes n})=0$ for all $n\geq1$.  By Remark \ref{bir} and Proposition \ref{rank2}, we have 
$H^0(\ts_3,\lambda_3(L\otimes K_X^{\otimes n}))=0=H^0(W(2,0,n),\lambda_{c_n^2}(L\otimes K_X^{\otimes n}))$ for all $n\geq 1.$  
Hence by (\ref{ses}), we have \begin{equation}\label{van}H^0(W(3,0,3),\lambda_3(L\otimes K_X))\cong H^0(W(3,0,3),\lambda_3(L\otimes K_X^{\otimes n}))~ for~ any ~n\geq1.\end{equation}
On the other hand, $\lambda_3(L\otimes K_X^{\otimes n})\cong \lambda_3(L)\otimes \lambda_3(K_X^{-1})^{\otimes -n}$ and $\lambda_3(K_X^{-1})$ is an effective line bundle associated to divisor $\ts_3$.  Hence $H^0(\lambda_3(L\otimes K_X^{\otimes n}))=0$ for $n$ large enough.  Therefore $H^0(W(3,0,3),\lambda_3(L\otimes K_X))=0$ by (\ref{van}).  Hence $\beta_D$ is an isomorphism and $SD_{3,L}$ is an isomorphism because so is $SD_{c_3^2,L}$ by Proposition \ref{rank2} (2).  Hence the theorem.
\end{proof}

We have the following theorem as a corollary to Theorem 1.4 (1) (2) in \cite{GY} by applying Corollary \ref{beta2} and Lemma \ref{like}.
\begin{thm}\label{rank3a}Let $r=3=n$.  $X$ and $L$ be as follows.
\begin{enumerate}
\item $X=\p^2$ or $\Sigma_e$ with $e=0,1$.  $L=-K_X$.

\item $X=\Sigma_e$ with $e=0,1$. $L=-K_X+F$ with $F$ the fiber class.
\end{enumerate}
Then $SD_{3,L}$ is an isomorphism under suitable polarization.
\end{thm}
\begin{proof}Let $(X,L)$ be as in the theorem.  By Theorem 1.4 (1) (2), $SD_{c^2_n,L}$ is an isomorphism for $n\geq 3$ under suitable polarization.  Hence the theorem follows from Lemma \ref{like} and Corollary \ref{beta2}.
\end{proof}

We have seen that $SD_{c_n^r,L}$ is an isomorphism for $r=1,2$, $n\geq r$ if either $H^0(L\otimes K_X)=0$ or $L=-K_X$ (see Corollary 4.3.2 in \cite{Yuan1} for $r=1$).  For $r\geq3$, $n\geq r$ and $(r,n)\neq (3,3)$, we only have a partial result as the following proposition. 
\begin{defn}\label{efsur1}The map $SD_{c,u}$ is called \emph{effectively surjective} (\emph{$\mu$-effectively surjective}, resp.) if we can find a finite collection $\big\{s_{\mg_i}\big\}_{i\in I}$ of sections with $\mg_i\in M^H_X(c)$ ($\mg_i\in\kw^H_X(c)^{\mu}$, resp.) spanning $H^0(M^H_X(u),\lambda_u(c))$, where $s_{\mg}$ is the section of $\lambda_u(c)$ induced by the following divisor
\[D_{\mg}:=\big\{\mf\in M^H_X(c)|H^0(\mg\otimes\mf)\neq0\big\}.\] 
\end{defn}
\begin{rem}\label{efsur2}By Proposition 6.18 in \cite{GY}, $SD_{c,u}$ is effectively surjective $\Rightarrow$ $SD_{u,c}$ is surjective.  Moreover by Lemma \ref{mustable}, for $n>r$ $SD_{c^r_n,L}$ is $\mu$-effectively surjective $\Rightarrow$ $SD_{c^r_n,L}$ is surjective.  
\end{rem}
\begin{prop}\label{part1}If either $H^0(L\otimes K_X)=0$ or $L= -K_X$, then $SD_{c^r_n,L}$ is $\mu$-effectively surjective for all $r\geq1$ and $n\geq r$.  Moreover $SD_{c^r_n,L}$ is effectively surjective for $r=1$ or $n=r$.  Therefore $SD_{c^r_n,L}$ is surjective for all $r\geq1$ and $n\geq r$.
\end{prop}
\begin{proof}The strategy is analogous to \S 6.2 in \cite{GY}.  
Let $L=-K_X$.  By Theorem 4.4.1 (2) in \cite{Yuan1} we have the surjective multiplication map for all $r\geq1$
\[m_1:~H^0(\ml,\z^r_L(r))\otimes H^0(\ml,\pi^{*}\mo_{\ls}(n-r))\ra H^0(\ml,\z_L^r(n)).\]
By analogous argument to \S 6.2 in \cite{GY}, we see that for all $r\geq 1$, if $\exists~\{\mg^r_i\}_{i\in I_r}\subset \wrr$ such that $\{s_{\mg_i^r}\}_{i\in I_r}$ spans $H^0(\ml,\z^r_L(r))$, then $\exists~\{\mg^{r,n}_i\}_{i\in I_{r,n}}\subset \kw(r,0,n)^{\mu}$ such that $\{s_{\mg_i^{r,n}}\}_{i\in I_{r,}}$ spans $H^0(\ml,\z^r_L(n))$ for all $n\geq r$, hence then $SD_{c^r_n,L}$ is surjective.  So we have the following implication
\begin{equation}\label{im1} SD_{r,L}~\text{{\small is effectively surjective}}\Longrightarrow SD_{c^r_n,L}~\text{{\small is}}~\mu\text{{\small-effectively surjective for all}}~n\geq r.\end{equation}

On the other hand, by Lemma \ref{like}, $s_{\mg^r}$ spans $H^0(D_{\z_L},\z^r_L(r)|_{D_{\z_L}})$ for any $\mg^r\in \wrr\setminus\ts_r$ via the map $\beta_D$ in (\ref{rsddim}).  Let $\ell=dim~\ls$ and choose $\ell$ distinct points $x_1,x_2,\cdots,x_{\ell}$ such that $t_{x_j}$ spans $H^0(\ls,\mo_{\ls}(1))$, where $t_{x_j}$ is the section of $\mo_{\ls}(1)$ induced by asking curves passing through $x_j$.  If $\exists~\{\mg^r_i\}_{i\in I_r}\subset \wrr$ such that $\{s_{\mg_i^r}\}_{i\in I_r}$ spans $H^0(\ml,\z^r_L(r))$, then we can construct a set $\{\mg^{r+1}_{i,x_j},\mg^{r+1}\}_{i\in I_r,1\leq j\leq l} \subset W(r+1,0,r+1)$ such that $\mg^{r+1}\in W(r+1,0,r+1)\setminus \ts_{r+1}$ and $\mg_{i,x_j}^{r+1}$ lies in the following sequence
\begin{equation}\label{basis1}0\ra \mg^r_i\ra\mg^{r+1}_{i,x_j}\ra \mi_{x_j}\ra0,
\end{equation}
where $\mi_{x_j}$ is the ideal sheaf of $x_j$.  It is easy to see $\mg_{i,x_j}^{r+1}\in \ts_{r+1}$ and $\{s_{\mg_{i,x_j}^{r+1}}\}_{i\in I_r,1\leq j\leq \ell}$ spans the image of $\alpha_{\ts_{r+1}}$ in $H^0(\ml,\z^{r}_L(r+1))$ as in (\ref{rsddim}).  Also by (\ref{im1}) we have $\alpha_{\ts_{r+1}}$ is surjective since so is $SD_{c^r_{r+1},L}$.  Hence we have the following implication
\begin{equation}\label{im2}SD_{r,L}~\text{{\small is effectively surjective}}\Longrightarrow SD_{r+1,L}~\text{{\small is effectively surjective}}.\end{equation}
$SD_{1,L}$ is surjective by Corollary 4.3.2 in \cite{Yuan1}.  Combining (\ref{im1}) and (\ref{im2}), we have proved the proposition for $L=-K_X$.

If $L\neq -K_X$, then by Proposition 4.1.1 in \cite{Yuan1}, $\ml\cong\ls$ and $\z_L\cong\mo_{\ls}$.  Use the same argument as Proposition 3.8 in \cite{Yuan5} and we get that $SD_{r,L}$ is effectively surjective.  By the analogous argument as \S 6.2 in \cite{GY}, we also have implication in (\ref{im1}) and hence the proposition. 
\end{proof}
\begin{rem}\label{S3}Proposition \ref{part1} holds not only for $(X,L)$ in Corollary \ref{beta2} but also for all $L$ on $X=\p^2$ or $\Sigma_e$ such that $-L-K_X$ is effective.  In particular by Lemma 3.17 and Corollary 3.19 in \cite{Yuan6}, Theorem 4.4.1 (2) in \cite{Yuan1} also applies to $L=-K_X$, $X=\Sigma_e$ for all $e\geq 2$, i.e. $\pi_{*}\z^r_L\cong \mo_{\ls}\oplus\displaystyle{\bigoplus_{i=2}^r}\mo_{\ls}(-i)$.
\end{rem}

\subsection{More results on $X=\p^2$.}\qquad

Using Fourier transform on $\p^2$ (see also \S 4 in \cite{LP1} or \S 3 in \cite{Yuan5}), we can get results as follows.  
\begin{thm}\label{p2}Let $X=\p^2$, $L=dH$.  Then
\begin{enumerate}
\item $g_r^{\vee}$ in (\ref{rsddim}) is injective for all $r>0$ and $d>0$;
\item $SD_{r,dH}$ is an isomorphism for $d=1,2$ and $r>0$;
\item $SD_{c_r^{r-1},dH}$ is an isomorphism for $d=1,2$ and $r>1$;
\item $SD_{3,rH}$ is injective for $r>0$;
\item $SD_{c_3^2,rH}$ is injective for all $r>0$. 
\end{enumerate}
\end{thm}
\begin{proof}[Proof of Statement (1) in Theorem \ref{p2}]To show $g_r^{\vee}$ is injective, it is enough to show $H^1(\lambda_r(H^{\otimes (d-3)}))=0$ for all $r>0$ and $d>0$.  Notice that $W(r,0,r)$ is of weight zero (def. see \S 1.2 in \cite{Dr2}, or \S 2.3 in \cite{Abe2}).  Hence by Theorem B, Theorem E and Theorem F in \cite{Dr2}, we have that $\Pic(W(r,0,r))\cong\mathbb{Z}$ is generated by $\lambda_r(H)$ and the dualizing sheaf over $\wrr$ is $\lambda_r((H^{-1})^{\otimes 3r})$.  Since $\lambda_r(H)$ is effective, it is an ample generator of $\Pic(\wrr)$.  By the generalized version of Kodaira vanishing theorem (Theorem 1-2-5 in \cite{KMM}), we have $H^1(\lambda_r(H^{\otimes(d-3r)}))=0$ for all $d>0$.   Therefore $H^1(\lambda_r(H^{\otimes (d-3)}))=0$ for all $r>0$ and $d>0$, hence Statement (1).
\end{proof}
To prove Statement (2) in Theorem \ref{p2}, we need to use Fourier transform on $\p^2$.  Let $\mathcal{D}$ be the universal curve in $\p^2\times |H|$ as follows.
\begin{equation}\xymatrix@C=0.01cm{
  \p^2\times |H|~~~~~\supset&\mathcal{D}\ar[d]^q\ar[r]^p
                & \p^2 \\
               &~~~~~~ |H|\cong\p^2&
               }.
\end{equation}

Let $\mf$ be a pure 1-dimensional sheaf with Euler characteristic 0, then its Fourier transform is defined to be $\mg_{\mf}:=q_{*}(p^{*}(\mf\otimes\mo_{\p^2}(2)))\otimes\mo_{|H|}(-1)$.  Let $\mg$ be a torsion free sheaf on $|H|$ with first Chern class 0 and Euler characteristic 0, then its Fourier transform is defined to be $\mf_{\mg}:=R^1p_{*}(q^{*}(\mg\otimes\mo_{|H|}(-1)))\otimes\mo_{\p^2}(-1)$.  
Identify $|H|$ with $\p^2$ and Fourier transform gives a birational correspondence
\begin{equation}\label{bcor}\Phi:\md \dashrightarrow W(d,0,d).\end{equation}
By Lemma 4.2 and Corollary 4.3 in \cite{LP1}, $\Phi$ induces an isomorphism from $\md\setminus \dzd$ to the open subset of $W(d,0,d)$ consisting of polystable sheaves whose restrictions to a generic line $\p^1\cong l\in|H|$ are isomorphic to $\mo_l^{\oplus d}$.  Moreover $\Phi$ is an isomorphism for $d=1,2$. 

Let $\md^g$ be the largest open subset where $\Phi$ is well-defined, i.e. $\md^g$ consists of sheaves $\mf$ such that $q_{*}(p^{*}(\mf\otimes\mo_{\p^2}(2)))\otimes\mo_{|H|}(-1)$ are semistable.  By Theorem 4.4 and Theorem 4.8 in \cite{LP1}, $\md^g=\md$ for $d\leq 4$.  For $d\geq 5$, the following lemma shows that the correspondence $\Phi$ can be defined over a larger open subset than $\md\setminus\dzd$.

\begin{lemma}\label{ftg0}Let $\mf\in\dzd$ with $d\geq 5$ such that $h^0(\mf)=h^1(\mf)=1$.  If the non-split extension of $\mf$ by $K_X$ is torsion-free, then $\mf\in\md^g$.  
In particular, if $\mf$ supports on an integral curve, then $\mf\in\md^g$. 
\end{lemma}
\begin{proof}$\Ext^1(\mf,\mo_{\p^2}(-3))\cong H^1(\mf)^{\vee}$.  Hence there is a unique non-split extension of $\mf$ by $\mo_{\p^2}(-3)$ as follows.
\begin{equation}\label{e1}0\ra \mo_{\p^2}(-3)\ra \widetilde\mi\ra \mf\ra 0.
\end{equation}
Do Fourier transform to (\ref{e1}) and we get 
\begin{equation}\label{e2}0\ra q_{*}(p^{*}(\widetilde\mi\otimes\mo_{\p^2}(2)))\otimes\mo_{|H|}(-1)\xrightarrow{\cong}q_{*}(p^{*}(\mf\otimes\mo_{\p^2}(2)))\otimes\mo_{|H|}(-1)\ra 0,\end{equation}
which is because $q_{*}(p^{*}(\mo_{\p^2}(-1)))=0=R^1q_{*}(p^{*}(\mo_{\p^2}(-1)))$.

$h^0(\widetilde\mi)\cong h^0(\mf)=1$, hence there is a non-zero map $\mo_{\p^2}\ra \widetilde\mi$ which is injective given $\widetilde\mi$ torsion-free.  Hence we have 
\begin{equation}\label{e3}
0\ra \mo_{\p^2}\ra \widetilde\mi\ra \mf_1\ra 0.
\end{equation}
$r(\mf_1)=0$, $c_1(\mf_1)=c_1(\widetilde{\mi})=(d-3)H$, $\chi(\mf_1)=\chi(\mf)=0$ and moreover $h^0(\mf_1)=h^0(\widetilde\mi)-1=0$ which implies that $\mf_1$ is semistable, since every subsheaf of $\mf_1$ can not have positive Euler characteristic.  Hence $\mf_1\in M((d-3)H,0)\setminus D_{\z_{(d-3)H}}$.

Let $\mg_1$ be the Fourier transform of $\mf_1$, then $\mg_1\in W(d-3,0,d-3)$.  Do Fourier transform to (\ref{e3}) and we get
\begin{equation}\label{e4}
0\ra S^2\mt_{\p^2}(-1) \ra q_{*}(p^{*}(\widetilde\mi\otimes\mo_{\p^2}(2)))\otimes\mo_{|H|}(-1)\ra \mg_1\ra 0,
\end{equation}
where $S^2\mt_{\p^2}\cong q_{*}(p^{*}(\mo_{\p^2}(2)))$ is the $2^{\text{nd}}$ symmetric power of the tangent bundle $\mt_{\p^2}$.  $S^2\mt_{\p^2}(-1)\in W(3,0,3)$ and hence by (\ref{e2}) and (\ref{e4}), the Fourier transform of $\mf$ is semistable and hence $\mf\in\md^g$.
\end{proof}

Let $\wrr^g:=\Phi(M(rH,0)^g)$.  
\begin{lemma}\label{ftg1}The complement of $\md^g$ ($\wrr^g$, resp.) in $\md$ ($\wrr$, resp.) is of codimension $\geq 2$.
\end{lemma}
\begin{proof}$\md\setminus\md^g$ is of codimension $\geq2$ is by Theorem 4.17 and Proposition 5.5 in \cite{Yuan4} together with Lemma \ref{ftg0}.  

$\wrr^g$ contains all the sheaves whose restrictions on a generic line $\p^1\cong l\in|H|$ are isomorphic to $\mo_l^{\oplus d}$.  It is enough to show that the following set $\tb$ is of codimension $\geq 2$ in $\wrr$.
\[\tb:=\big\{\mg\in\wrr| H^0(\mg\otimes\mo_l(-1))\neq 0, ~\forall~l\in |H|.\big\}\]

Let $\widehat{H}$ be the subspace of $H^0(\wrr,\lambda_r(H))$ generated by all the sections induced by sheaves $\mo_l(-1)$ with $l\in |H|$. Also $\widehat{H}$ is the image of $H^0(M(H,0),\lambda_H(r))^{\vee}\cong H^0(|H|,\mo_{|H|}(r))^{\vee}$ via the strange duality map $SD_{H,r}$. Notice that $\tb$ is the base locus of $\widehat{H}$.  

Since $\lambda_r(H)$ is the ample generator of $\Pic(\wrr)$, every divisor in $|\lambda_r(H)|$ can not be a union of two subdivisors.  Hence either $\tb$ is of codimension $\geq 2$ in $\wrr$, or $dim~\widehat{H}=1$. 
By Proposition \ref{part1}, $SD_{H,r}$ is injective and hence $dim~\widehat{H}=h^0(|H|,\mo_{|H|}(r))\geq 3$.  Hence $\wrr\setminus\wrr^g$ is of codimension $\geq 2$.
Hence the lemma. 
\end{proof}

\begin{lemma}\label{ftg2}$\Phi:~\md\ra W(d,0,d)$ is a birational map of normal projective schemes and $\Phi^{*}\lambda_d(H^{\otimes r})\cong \lambda_{dH}(r)=\z^r_{dH}(r)$, $\forall~ r$.  

Moreover $\Phi^{*}: H^0(W(d,0,d),\lambda_d(H^{\otimes r}))\xrightarrow{\cong} H^0(\md,\z^r_{dH}(r))$ is an isomorphism.
\end{lemma}
\begin{proof}$\Phi$ is a birational map not only set-theoretically but also schematically, because of Lemma 3.3 (1) in \cite{Yuan5}.  Also by Proposition 3.6 in \cite{Yuan5},  $\Phi^{*}\lambda_d(H^{\otimes r})\cong \lambda_{dH}(r)=\z^r_{dH}(r).$ 

Since both $\md$ and $W(d,0,d)$ are normal and irreducible, $\Phi_{*}\mo_{\md^g}\cong \mo_{W(d,0,d)^g}$.  Therefore
$\Phi_{*}\z_{dH}^r(r)\cong \lambda_d(H^{\otimes r})$.  By Lemma \ref{ftg1} we have
\[\begin{array}{c}H^0(\md,\z^r_{dH}(r))\xrightarrow{\cong} H^0(\md^g,\z^r_{dH}(r))\xrightarrow{\cong}H^0(W(d,0,d)^g,\Phi_{*}\z_{dH}^r(r))\\
\xrightarrow{\cong} H^0(W(d,0,d)^g,\lambda_d(H^{\otimes r}))\xrightarrow{\cong} H^0(W(d,0,d),\lambda_d(H^{\otimes r}))\end{array}.\]
The lemma is proved.
\end{proof}

\begin{proof}[Proof of Statement (2), (3), (4) and (5) in Theorem \ref{p2}]For $d=1,2$, by Proposition \ref{part1} $SD_{dH,r}$ is injective, hence $SD_{r,dH}$ is surjective.  To prove Statement (2), we only need to show that 
\begin{equation}\label{mw0}h^0(\wrr,\lambda_r(H^{\otimes d}))=h^0(|dH|,\mo_{|dH|}(r))~for~d=1,2.\end{equation} 

By Lemma \ref{ftg2}, \begin{equation}\label{mw1}h^0(\wrr,\lambda_r(H^{\otimes d}))=h^0(M(rH,0),\z_{rH}^d(d)).\end{equation}  
By Corollary 4.3.2 in \cite{Yuan1} and Theorem 1.1 in \cite{Yuan5}, we have
\begin{equation}\label{mw2}h^0(M(rH,0),\z_{rH}^d(d))=h^0(W(d,0,d),\lambda_d(H^{\otimes r}))~for~d=1,2.\end{equation}
By Fourier transform, 
\begin{equation}\label{mw3}h^0(W(d,0,d),\lambda_d(H^{\otimes r}))=h^0(\md,\z_{dH}^{r}(r))=h^0(|dH|,\mo_{dH}(r))~for~d=1,2.\end{equation}
Combining (\ref{mw1}), (\ref{mw2}) and (\ref{mw3}), we get (\ref{mw0}) and hence Statement (2).

Statement (3) is a direct consequence of Statement (1) (2), Corollary \ref{beta2} and Lemma \ref{like}.

By Corollary 3.7 in \cite{Yuan5} and Lemma \ref{ftg2} we have the following commutative diagram 
\begin{equation}\label{comm1}\tiny\xymatrix@C=0.5cm@R=1cm{H^0(M(rH,0)^g,\z_{rH}^d(d))^{\vee}\ar[r]^{\imath^{\vee}_{rH}}_{\cong}\ar[d]_{\cong} &H^0(M(rH,0),\z_{rH}^d(d))^{\vee}\ar[r]^{SD_{rH,d}=SD_{d,rH}^{\vee}} & H^0(W(d,0,d),\lambda_d(H^{\otimes r}))\ar[r]^{\jmath_d}_{\cong} &H^0(W(d,0,d)^g,\lambda_d(H^{\otimes r}))\ar[d]^{\cong}\\
H^0(W(r,0,r)^g,\lambda_r(H^{\otimes d}))\ar[r]_{\jmath_r^{\vee}}^{\cong}&H^0(W(r,0,r),\lambda_r(H^{\otimes d}))^{\vee}\ar[r]_{SD_{r,dH}=SD_{dH,r}^{\vee}} &H^0(\md,\z^r_{dH}(r))\ar[r]_{\imath_{dH}}^{\cong}&H^0(M(dH,0)^g,\z^r_{dH}(r)). }\normalsize
\end{equation}

By Proposition \ref{part1} $SD_{r,3H}$ is surjective, hence so is $SD_{rH,3}=SD_{3,rH}^{\vee}$ by (\ref{comm1}) and hence $SD_{3,rH}$ is injective.

Statement (5) is a direct consequence of Statement (1) (4) and Corollary \ref{beta2}.  We have proved the theorem. 
\end{proof}
\begin{rem}\label{friday}By Corollary \ref{beta2} and Theorem \ref{p2} (1), $SD_{3,rH}$ is an isomorphism for all $r>0$ $\Leftrightarrow$ $SD_{c_3^2,rH}$ is an isomorphism for all $r>0$.  
\end{rem}

\section{Another result on $\p^2$ and $\Sigma_e$ with $e\leq 3$.}
\subsection{Statements.}\qquad

In this section we choose $L$ to be some special cases where the arithmetic genus $g_L$ of curves in $\ls$ is 3.  The main result is as follows.

\begin{prop}\label{gth1}\begin{enumerate}\item $X=\p^2$ and $L=4H$.  Then 
$$\pi_{*}\z_{L}^2\cong \mo_{\ls}\oplus\mo_{\ls}(-2)^{\oplus 6}\oplus \mo_{\ls}(-3).$$
\item $X=\Sigma_e$ with $e\leq 3$ and $L=2G+(e+4)F$ where $F$ is the fiber class and $G$ is the section class such that $G.G=-e$.  Then 
$$\pi_{*}\z_{L}^2\cong\mo_{\ls}\oplus\mo_{\ls}(-2)^{\oplus 6}\oplus \mo_{\ls}(-4).$$
\end{enumerate}
\end{prop}

\begin{coro}\label{gth2}Let $(X,L)$ be as in Proposition \ref{gth1} and let $-K_X$ be ample (i.e. $X\neq \Sigma_e$ with $e=2,3$), then for any $n\geq 3$ under suitable polarization, 
{\footnotesize\begin{equation}\label{ncgth} h^0(\ml,\z^2_L(n))=\chi(\ml,\z^2_L(n))=h^0(W(2,0,n),\lambda_{c^2_n}(L))=\chi(W(2,0,n),\lambda_{c^2_n}(L)).
\end{equation}}
\end{coro}
\begin{proof}The corollary follows straightforward 
from Theorem 1.3 (1) in \cite{G1} (for $X=\p^2$), Theorem 1.2 (3) (for $X=\Sigma_e$), Proposition 2.9 in \cite{GY} and Proposition \ref{gth1} above.  
\end{proof}

\begin{thm}\label{gth3}Let $(X,L)$ be as in Proposition \ref{gth1} and let $-K_X$ be ample, then under suitable polarization $SD_{c_n^2,L}$ is an isomorphism for any $n\geq3$.
\end{thm}

\begin{rem}\label{gth4}\begin{enumerate}\item If $r=n=2$, $SD_{2,L}$ is an isomorphism by \cite{Yuan5} and \cite{Yuan6}.  For $X=\p^2$, we also have equality in (\ref{ncgth}), but for $X=\Sigma_e$ we still don't know whether $\chi(\ml,\z^2_L(2))=\chi(W(2,0,2),\lambda_2(L))?$  

\item By Theorem 1.2 (4) in \cite{GY}, Corollary \ref{gth2} and hence Theorem \ref{gth3} are true under any polarization for $n$ very large.
\end{enumerate}
\end{rem}

\begin{coro}\label{gth5}Let $(X,L)$ be as in Proposition \ref{gth1} and let $-K_X$ be ample, then under suitable polarization $SD_{3,L}$ is an isomorphism.
\end{coro}
\begin{proof}This follows directly from Corollary 4.3.2 in \cite{Yuan1}, Corollary \ref{beta2} and Proposition \ref{gth1}.
\end{proof}

Let $(X,L)$ be as in Proposition \ref{gth1}.  On $M(L,0)$ we have an exact sequence similar as (\ref{zes})
\begin{equation}\label{g3e1}
0\ra\z_{L}\ra\z^{2}_{L}\ra \z^{2}_L|_{D_{\z_{L}}}\ra0.
\end{equation}
Push it forward via $\pi$ to $\ls$.  By Corollary 3.19 (2) in \cite{Yuan6},  we have $\pi_{*}\z_{L}\cong\mo_{\ls}$ and $R^i\pi_{*}\z^r_{L}=0$ for all $i,r>0$.  Hence 
\begin{equation}\label{g3e2}
0\ra\mo_{\ls}\ra\pi_{*}\z_L^2\ra \pi_{*}(\z_L^2|_{D_{\z_L}})\ra0.
\end{equation}
Then Proposition \ref{gth1} is just a direct corollary of the following lemma.
\begin{lemma}\label{gthlem}
\begin{enumerate}
\item $X=\p^2$ and $L=4H$.  Then 
$$\pi_{*}(\z_{L}^2(2)|_{D_{\z_{L}}})\cong \mo_{\ls}^{\oplus 6}\oplus \mo_{\ls}(-1).$$
\item $X=\Sigma_e$ with $e\leq 3$ and $L=2G+(e+4)F$.  Then 
$$\pi_{*}(\z_{L}^2(2)|_{D_{\z_{L}}})\cong \mo_{\ls}^{\oplus 6}\oplus \mo_{\ls}(-2).$$
\end{enumerate}
\end{lemma} 

We will prove Lemma \ref{gthlem} and Theorem \ref{gth3} in \S \ref{p2g3} for $X=\p^2$ and \S \ref{hg3} for $X=\Sigma_e$ with $e\geq3$.

Let $\mc_{L}\subset \p^2\times |L|$ be the universal curve and let $\mc_{L}^{[n]}$ be the relative Hilbert scheme of $n$-points on $\mc_{L}$ over $|L|$.  We have two morphisms $\bar{\pi}:\mc_L^{[n]}\ra \ls$ sending each $[Z\subset C]$ to the curve $C$, and $\rho:\mc_{L}^{[n]}\ra X^{[n]}$ sending $[Z\subset C]$ to $Z$, where $X^{[n]}$ is the Hilbert scheme of $n$-points on $X$.  For each line bundle $\me$ over $X$, denote by $\me_{[n]}$ the determinant line bundle $det^{-1}R^{\bullet}p_{*}(q^{*}\me\otimes \I_n)$ over $X^{[n]}$, where $\I_n$ is the universal ideal sheaf over $X\times X^{[n]}$, and denote by $\me_{(n)}$ the line bundle over $X^{[n]}$ induced by the $\ks_n$-linearized line bundle $\me^{\boxtimes n}$ over $X^{ n}$.  It is well known (e.g. see \S 5 in \cite{EGL}) that 
\begin{equation}\label{egl}\me_{[n]}\cong\me_{(n)}\otimes\mo_{X^{[n]}}(-\Delta/2),\end{equation} 
where $\Delta$ is the exceptional divisor of $X^{[n]}\ra Sym^n X$, i.e. $\Delta$ consists of all $Z$ supported at at most $n-1$ points.

By \S 4 in \cite{Yuan5} or the proof of Lemma 3.9 in \cite{Yuan6}, we see that there is a birational map $g:Q:=\mc_{L}^{[g_L-1]}\dashrightarrow D_{\z_L}$, defined by assigning each $[Z\subset C]\in \mc_{L}^{[g_L-1]}$ to $\mi_{Z/C}(L\otimes K_X)$ with $\mi_{Z/C}$ the ideal sheaf of $Z$ inside the curve $C\in\ls$.  Moreover, by Lemma 3.7 in \cite{Yuan4} $g$ induces an isomorphism between $Q^o$ and $D_{\z_L}^o$ defined as follows
$$Q^o:=\left\{[Z\subset C]\in \mc_{L}^{[g_L-1]}\left|\begin{array}{l}h^0(\mi_Z(L\otimes K_X))=1,~ h^1(\mi_Z(L))=0,\\ 
and~C~is~integral.
\end{array}\right.\right\},$$

$$D_{\z_L}^o:=\left\{\mf\in D_{\z_L}\left|\begin{array}{l}h^1(\mf)=1,~ h^1(\mf(-K_X))=0,\\ 
and~Supp(\mf_L)~is~integral.
\end{array}\right.\right\}.$$
$D_{\z_L}^o$ is dense in $D_{\z_L}$ by $\cb$-(1) in \cite{Yuan6}.  Also $g^{*}(\z_L^r(r))\cong\rho^{*}(L\otimes K_X)^{\otimes r}_{[g_L-1]}$.  Define $\cl_{L}:=(L\otimes K_X)_{[g_L-1]}$.  We have the commutative diagram
\begin{equation}\label{comm2}\xymatrix{Q \ar@{-->}[r]^{g}\ar[rd]_{\bar{\pi}}& D_{\z_L}\ar[d]^{\pi}\\ &\ls}.
\end{equation} 

We have a useful lemma as follows.
\begin{lemma}\label{use}Let $X$ be any rational surface and $L$ an effective line bundle.  Let $Z\in X^{[g_L-1]}$. If $h^0(\mi_Z(L\otimes K_X))=1$ with $\mi_Z$ the ideal sheaf of $Z$, then for any non-zero map $\kappa:K_X\ra \mi_Z(L\otimes K_X)$ we have the following exact sequence
\begin{equation}\label{kappa}0\ra K_X\xrightarrow{\kappa}\mi_Z(L\otimes K_X)\ra \mf_L\ra0,\end{equation}
with $\mf_L\in\ml$.
\end{lemma}
\begin{proof}$\kappa$ has to be injective since it is non-zero.  $\mf_L$ is pure because $\mi_Z(L\otimes K_X)$ is torsion free.
$h^0(\mf_L)=h^0(\mi_Z(L\otimes K_X))=1$, hence for any $\mf'\subsetneq \mf_L$, $h^0(\mf')\leq1$.  If $\mf_L$ is not semistable, then $\exists~ \mf'\subsetneq \mf_L$, such that $\chi(\mf')>0$ and hence that $h^1(\mf')=1$.  Therefore (\ref{kappa}) must partially split along $\mf'$ which contradicts with the torsion-freeness of $\mi_Z(L\otimes K_X)$.  So the lemma is proven.
\end{proof}

\subsection{Proof for $X=\p^2$ and $L=4H$.}\label{p2g3}\qquad

In this subsection $L=4H$ and $g_L-1=2$. 
Denote by $\mc$ instead of $\mc_{4H}$ the universal curve for simplicity.
Since (\ref{comm2}) commutes, Lemma \ref{gthlem} for $X=\p^2$ follows from the following two lemmas.
\begin{lemma}\label{p2z1}The birational map $g:Q\dashrightarrow D_{\z_{4H}}$ is a morphism and $g_{*}\mo_{Q}\cong \mo_{D_{\z_{4H}}}$, $R^ig_*\mo_Q=0$ for all $i>0$ and $i\neq2$.  

Moreover $R^i\bar{\pi}_{*}\cl_{4H}^r\cong R^{i-2}\pi_{*}(R^2g_{*}\mo_Q\otimes\z_{4H}^r(r)|_{D_{\z_{4H}}})$ for all $i>0$ and $r\geq2$, in particular $R^1\bar{\pi}_{*}\cl_{4H}^r=0$ for all $r\geq 2$.
\end{lemma}
\begin{lemma}\label{p2z2}$\bar{\pi}_{*}\cl_{4H}^2\cong\mo_{|4H|}^{\oplus 6}\oplus\mo_{|4H|}(-1)$.
\end{lemma}
\begin{proof}[Proof of Lemma \ref{p2z1}]$\forall~Z\in (\p^2)^{[2]}$, $h^0(\mi_Z(1))=1$.  Hence $g$ is well-defined over the whole $Q$ by Lemma \ref{use}.  We have
\[0\ra \mo_{\p^2}\ra\mi_Z(1)\ra \mo_{\p^1}(-1)\ra 0,~\forall~Z\in~(\p^2)^{[2]}.\]
Hence $h^0(\mi_Z(4))=13$, $h^1(\mi_Z(4))=0$ and $Q$ is a $\p^{12}$-bundle over $(\p^2)^{[2]}$.  Both $Q$ and $D_{\z_L}$ are projective and $g$ is dominant.  Hence $g$ must be surjective.  We have the following sequence 
\begin{equation}\label{p2coh1}0\ra \mo_{\p^2}(-3)\xrightarrow{\kappa}\mi_Z(1)\ra\mf\ra0,\end{equation}
where $h^0(\mf)=h^0(\mi_Z(1))=1=h^1(\mf)$.  Hence if $\mf$ is stable, $g^{-1}(\mf)$ contains only one element.  Let $D_{\z_{4H}}^s$ be the stable locus of $D_{\z_{4H}}$, then $g$ is an isomorphism over $D_{\z_{4H}}^s$.

If $\mf$ in (\ref{p2coh1}) is strictly semistable, then $Supp(\mf)$ must be reducible.  Write $Supp(\mf)=C_1\cup C_3$, then $\p^1\cong C_1\in |H|$, $C_3\in |3H|$ and $\mf$ is $S$-equivalent to $\mo_{C_3}\oplus\mo_{C_1}(-1)$ since $H^0(\mf)\neq 0$.  This happens iff the map $\kappa$ in (\ref{p2coh1}) factors through $\mo_{\p^2}\hookrightarrow \mi_Z(1)$.  Hence the fiber at $\mo_{C_3}\oplus\mo_{C_2}$ of $g$ is $C_1^{[2]}\cong Sym^2(C_1)\cong\p^2$.  Hence $R^ig_*\mo_Q=0$ for all $i>2$. 

$dim~(D_{\z_{4H}}\setminus D_{\z_{4H}}^s)=dim~|3H|+dim |H|=11=dim~D_{\z_{4H}}-5.$. Thus $D_{\z_{4H}}$ is normal and $g_{*}\mo_Q\cong \mo_{D_{\z_{4H}}}$. 
Let $Q\supset T:=g^{-1}(D_{\z_{4H}}\setminus D_{\z_{4H}}^s)$.  Then $dim~T=dim~Q-3.$ Since $g$ is an isomorphism outside $T$, $R^1g_*\mo_Q=0$. 

Since $R^i\pi_{*}\z_L^r=0$ for all $i,r>0$, $R^i\pi_{*}(\z_L^r|_{D_{\z_L}})=0$ for all $i>0$ and $r\geq2$.  Hence 
by spectral sequence $R^i\bar{\pi}_{*}\cl_{4H}^r\cong R^{i-2}\pi_{*}(R^2g_{*}\mo_Q\otimes\z_{4H}^r(r)|_{D_{\z_{4H}}})$ for all $i>0$ and $r\geq2$.  In particular $R^1\bar{\pi}_{*}\cl_{4H}^r=0$ for all $r\geq2$.
\end{proof}

\begin{proof}[Proof of Lemma \ref{p2z2}]The map $\rho:Q\ra (\p^2)^{[2]}$ is a $\p^{12}$-bundle, so $Q$ is smooth.  By (\ref{egl}) we have 
\begin{equation}\label{egl2}\cl_{4H}^2\cong \rho^{*} H^{\otimes 2}_{(2)}\otimes \rho^{*}\mo_{(\p^2)^{[2]}}(-\Delta)\cong \rho^*H^{\otimes2}_{(2)}\otimes \mo_Q(-\Delta_Q).\end{equation}  
where $Q\supset\Delta_{Q}:=\rho^{*}\Delta$.  We have the Hilbert-Chow map $h:Q\ra Sym^2_{|4H|}\mc$ which is an isomorphism on $Q\setminus \Delta_{Q}$ and over all $[\{2x\}\subset C]$ such that $C$ are smooth at $x$.  
We have the commutative diagram as follows
\begin{equation}\label{comm3}\xymatrix{&Q \ar[r]^{\rho}\ar[ld]_{\bar{\pi}}\ar[d]^{h}& (\p^2)^{[2]}\ar[d]^{h_1}\\ |4H|&Sym^2_{|4H|}\mc\ar[l]^{\bar{\pi}_1}\ar[r]_{\rho_1}&Sym^2\p^2}.
\end{equation} 
Let $\Delta_{\mc}\subset Sym^2_{|4H|}\mc$ be the diagonal.  Then $\Delta_{\mc}\cong \mc$ and $h^*\Delta_{\mc}=\Delta_Q$. 
Notice that $H_{(2)}$ over $(\p^2)^{[2]}$ is the pull-back via $h_1$ the line bundle denoted also by $H_{(2)}$ over $Sym^2\p^2$.  Hence by (\ref{egl2}) we have
\begin{equation}\label{egl3}\cl_{4H}^2\cong h^*(\rho_1^{*} H^{\otimes 2}_{(2)}\otimes \mo_{Sym^2_{|4H|}\mc}(-\Delta_{\mc})).\end{equation}  
Define $\ma:=\rho_1^{*} H^{\otimes 2}_{(2)}\otimes \mo_{Sym^2_{|4H|}\mc}(-\Delta_{\mc})$.  $Sym^2_{|4H|}\mc$ is normal and hence $h_{*}\mo_Q=\mo_{Sym^2_{|4H|}\mc}$.  It is enough to show 
\begin{equation}\label{p2pl1}(\bar{\pi}_1)_{*}\ma\cong\mo_{|4H|}^{\oplus 6}\oplus\mo_{|4H|}(-1).\end{equation}

We also have the following commutative diagram
\begin{equation}\label{p2com1}\xymatrix{\Delta_{\mc}\ar@{^{(}->}[r] \ar[d]_{\cong}&\mc\times_{|4H|}\mc\ar@{^{(}->}[r]\ar@/^2pc/[rrr]_{q_2}\ar@/^3pc/[rrr]^{q_1}\ar[d]_{\sigma} &\p^2\times \p^2\times |4H|\ar[r] &\p^2\times \p^2\ar[d]^{\sigma_1}\ar@<0.5ex>[r]^{\widetilde{q}_1}\ar@<-0.5ex>[r]_{\widetilde{q}_2}&\p^2\\ 
\Delta_{\mc}\ar@{^{(}->}[r] &Sym^2_{|4H|}\mc\ar[rr]_{\rho_1} & & Sym^2\p^2&.}
\end{equation}
$\sigma_1^{*}H_{(2)}\cong \mo_{\p^2}(1)^{\boxtimes 2}$.  We then have
\begin{equation}\label{egl3}\sigma^{*}\ma\cong q_1^{*}\mo_{\p^2}(2)\otimes q_2^{*}\mo_{\p^2}(2)\otimes \mo_{\mc\times_{|4H|}\mc}(-\Delta_{\mc}).\end{equation}

On $\mc\times_{|4H|}\mc$ we have 
\begin{equation}\label{p2ex1}
0\ra\sigma^*\ma\ra  q_1^{*}\mo_{\p^2}(2)\otimes q_2^{*}\mo_{\p^2}(2)\ra  q_1^{*}\mo_{\p^2}(2)\otimes q_2^{*}\mo_{\p^2}(2)|_{\Delta_{\mc}}\ra 0.\end{equation}
Define $\bar{\pi}_2=\bar{\pi}_1\circ~ \sigma$.  Then $(R^i(\bar{\pi}_2)_{*}(\sigma^*\ma))^{\ks_2}\cong R^i(\bar{\pi}_1)_*\ma$, where $\ks_2$ is the $2^{\text{nd}}$ symmetric group.  Since $R^1\bar{\pi}_*\cl^2_{4H}=0$ by Lemma \ref{p2z1}, $R^1(\bar{\pi}_1)_*\ma=0$.

$\ks_2$ acts on $\Delta_{\mc}$ trivially.  $q_1^{*}\mo_{\p^2}(2)\otimes q_2^{*}\mo_{\p^2}(2)|_{\Delta_{\mc}}\cong q^*\mo_{\p^2}(4)$ and hence $(\bar{\pi}_2)_{*}q_1^{*}\mo_{\p^2}(2)\otimes q_2^{*}\mo_{\p^2}(2)|_{\Delta_{\mc}}\cong p_{*}(q^*\mo_{\p^2}(4))$ with $p,q$ the projection of $\mc$ to $|4H|$ and $\p^2$ respectively.  We have on $\p^2\times |4H|$
\begin{equation}\label{p2uc}0\ra \mo_{\p^2}(-4)\boxtimes \mo_{|4H|}(-1) \ra \mo_{\p^2\times |4H|}\ra\mo_{\mc}\ra 0.
\end{equation}
Hence
\begin{equation}\label{p2dia}0\ra \mo_{|4H|}(-1)\ra \mo_{|4H|}\otimes H^0(\mo_{\p^2}(4))\ra p_{*}(q^*\mo_{\p^2}(4))\ra 0.
\end{equation}

$\mc\times_{|4H|}\mc$ is a complete intersection defined by a section of $\widetilde{q}_1^*\mo_{\p^2}(4)\otimes p^{*}\mo_{|4H|}(1)$ and a section of $\widetilde{q}_2^*\mo_{\p^2}(4)\otimes p^{*}\mo_{|4H|}(1)$ inside $\p^2\times\p^2\times |4H|$, where by abuse of notations $p$ is the projection from $\p^2\times\p^2\times |4H|$ to $|4H|$.   Hence on $\p^2\times\p^2\times |4H|$ we have the following exact sequence
\begin{equation}\label{p2pro}{\tiny 0\ra \left.\begin{array}{c}p^{*}\mo_{|4H|}(-2)\\ \otimes \\ \widetilde{q}_1^*\mo_{\p^2}(-4)\\ \otimes\\ \widetilde{q}_2^*\mo_{\p^2}(-4)\end{array}\right.\ra
\left.\begin{array}{c}\widetilde{q}_1^*\mo_{\p^2}(-4)\otimes p^{*}\mo_{|4H|}(-1)\\ \bigoplus\qquad~\\ \widetilde{q}_1^*\mo_{\p^2}(-4)\otimes p^{*}\mo_{|4H|}(-1)\end{array}\right.
\ra\mo_{\p^2\times\p^2\times|4H|}\ra\mo_{\mc\times_{|4H|}\mc}\ra0.}
\end{equation}
Therefore $(\bar{\pi}_2)_{*}(q_1^{*}\mo_{\p^2}(2)\otimes q_2^{*}\mo_{\p^2}(2))\cong \mo_{|4H|}\otimes H^0(\mo_{\p^2}(2))^{\otimes 2}$ and $\ks_2$ acts on $H^0(\mo_{\p^2}(2))^{\otimes 2}$ by switching two factors.  Hence $((\bar{\pi}_2)_{*}(q_1^{*}\mo_{\p^2}(2)\otimes q_2^{*}\mo_{\p^2}(2)))^{\ks_2}\cong \mo_{|4H|}\otimes S^2 H^0(\mo_{\p^2}(2))$.  Since $R^1(\bar{\pi}_1)_*\ma=0$, we then have 
\begin{equation}\label{p2pf}0\ra ((\bar{\pi}_2)_{*}(\sigma^*\ma))^{\ks_2}\ra \mo_{|4H|}\otimes S^2H^0(\mo_{\p^2}(2))\ra p_{*}(q^*\mo_{\p^2}(4))\ra 0.\end{equation}
Combine (\ref{p2dia}) and (\ref{p2pf}) and we get 
\begin{equation}\label{finp2}\xymatrix@C=0.6cm @R=0.5cm{&&&0\ar[d]&\\ &0\ar[d]&&\mo_{|4H|}(-1)\ar[d]&\\ 0\ar[r]&\mo_{|4H|}^{\oplus 6}\ar[r]\ar[d]&\mo_{|4H|}^{\oplus 21}\ar[r]^{\mathtt{r}'}\ar[d]^{\cong}& \mo_{|4H|}^{\oplus 15}\ar[r]\ar[d]&0\\
0\ar[r] &((\bar{\pi}_2)_{*}(\sigma^*\ma))^{\ks_2} \ar[d]\ar[r]&\mo_{|4H|}\otimes S^2H^0(\mo_{\p^2}(2))\ar[r]^{\qquad\mathtt{r}} & p_{*}(q^*\mo_{\p^2}(4))\ar[r]\ar[d] &0\\
&\mo_{|4H|}(-1)\ar[d]&&0&\\&0}.\end{equation}
Hence $(\bar{\pi}_1)_*\ma\cong ((\bar{\pi}_2)_{*}(\sigma^*\ma))^{\ks_2}\cong \mo_{|4H|}^{\oplus 6}\oplus \mo_{|4H|}(-1)$.
\end{proof}
\begin{rem}\label{rsurp2}We can see the surjectivity of map $\mathtt{r}$ in \ref{finp2} directly: the map $\mathtt{r}'$ in (\ref{finp2}) is given by the multiplication $$m:S^2H^0(\mo_{\p^2}(2))\twoheadrightarrow H^0(\mo_{\p^2}(4)),~(s_1,s_2)\mapsto s_1\cdot s_2,$$
which is surjective, hence $\mathtt{r}'$ is surjective and so is $\mathtt{r}$.
\end{rem}

\begin{proof}[Proof of Theorem \ref{gth3} for $X=\p^2$.]
Now we have the numerical condition $$h^0(M(4H,0),\z_{4H}^2(n))=h^0(W(2,0,n),\lambda_{c_n^2}(4H))~for ~n\geq3.$$  To show the strange duality map $SD_{c_n^2,4H}$ is an isomorphism, it is enough to show that it is surjective.  By Proposition \ref{gth1} the multiplication map 
\[m_2:H^0(M(4H,0),\z_{|4H|}^2(3))\otimes H^0(\pi^{*}\mo_{|4H|}(n-3))\twoheadrightarrow H^0(M(4H,0),\z_{|4H|}^2(n))\]
is surjective.  By analogous argument to \S6.2 in \cite{GY}, we only need to find $\big\{\mg_i\big\}_{i\in I}\subset \kw(2,0,3)^{\mu}$ such that the induced sections $\{s_{\mg_i}\}_{i\in I}$ (def. see Definition \ref{efsur1}) spans $H^0(M(4H,0),\z^2_{4H}(3))$, i.e. $SD_{c_3^2,4H}$ is $\mu$-effectively surjective.  

Let $H^0(M(4H,0),\z_{4H}(n))\hookrightarrow H^0(M(4H,0),\z^2_{4H}(n))$ with the embedding given by multiplying the section associated to $D_{\z_{4H}}$.  We can find sections $\{s_{\mg_j}\}_{j\in J}$ that spans $H^0(M(4H,0),\z_{4H}(n))$ (e.g. see Lemma 6.20 in \cite{GY}).  Hence we only need to find $\{s_{\mg_i}\}_{i\in I}$ to span $H^0(\z^2_{4H}(3)|_{D_{\z_{4H}}}).$  We firstly find $\mg_i\in W(2,0,2)$, $1\leq i\leq 6$ such that $\{s_{\mg_i}\}_{i=1}^6$ are linearly independent restricted to $D_{\z_{4H}}$, then $\{s_{\mg_i}\}_{i=1}^6$ spans $H^0(\z^2_{4H}(2)|_{D_{\z_{4H}}})\cong \mathbb{C}^6$ by Lemma \ref{gthlem}. 

$W(2,0,2)\cong |2H|$ by Fourier transform.  For every $\mg\in W(2,0,2)$ the support of its Fourier transform is a curve $C_{\mg}$ of degree 2 inside $|H|\cong \p^2$.  $C_{\mg}$ also consists of all the jumping lines $l$ of $\mg$, i.e. $[l]\in |H|$ such that $\mg|_{l}\not\cong \mo_{l}^{\oplus 2}$.  $C_{\mg}$ is integral $\Leftrightarrow$ $\mg$ is stable hence locally free by Lemma 2.3 in \cite{Yuan6}.  Choose generic 5 distinct points $l_1,l_2,\cdots,l_5\in |H|$,  then we can find a integral curve $C_6$ of degree 2 on $|H|$ passing through all these 5 points.  This can be done since $|\mo_{|H|}(2)|\cong \p^5$.   Then we can find 5 other integral curve $C_1,C_2,\cdots,C_5$ of degree 2 over $|H|$ such that $l_i\in C_j\Leftrightarrow i<j$.  So $C_1$ contains none of those 5 points.  Let $\mg_i$ be the Fourier transform of $\mo_{C_i}(-1)$, then $\mg_i\in W(2,0,2)$ and are locally free.  We claim that $\{s_{\mg_i}\}_{i=1}^6$ are linearly independent restricted to $D_{\z_{4H}}$.  To show the claim we choose $Z_i\in (\p^2)^{[2]}$ such that $Z_i\in l_i$, i.e. $l_i=l_{Z_i}$ (def. see Lemma \ref{p2ind}) for all $1\leq i\leq 5$.  Let $\mf_i\in D_{\z_{4H}}$, $1\leq i\leq 5$ lie in the following sequence
\[0\ra \mo_{\p^2}(-3)\ra \mi_{Z_i}(1)\ra \mf_i\ra 0.\]
Then by Lemma \ref{p2ind} below, we have $s_{\mg_j}(\mf_i)=0\Leftrightarrow i<j$.  If $\exists ~a_1,\cdots,a_6\in \mathbb{C}$, such that $s:=\displaystyle{\sum_{i=1}^6}a_is_{\mg_i}=0$, then $s(\mf_1)=a_1s_{\mg_1}(\mf_1)=0$ hence $a_1=0$ for $s_{\mg_1}(\mf_1)\neq 0$.  Taking value of $s$ on $\mf_2,\cdots,\mf_5$ and we get $a_2=\cdots=a_5=0$.  Hence $a_6=0$ and $s_{\mg_i}$ are linearly independent restricted to $D_{\z_{4H}}$.

Since $\{s_{\mg_i}\}_{i=1}^6$ spans $H^0(\z^2_{4H}(2)|_{D_{\z_{4H}}})$, as what we did before, we can find $\{s_{\mg_{i,x_k}}\}_{i,k}$ spans the image of the map $m_3$ as follows.
\begin{equation}\label{p2m3}H^0(\pi^{*}\mo_{|4H|}(1))\otimes H^0(\z^2_{4H}(2)|_{D_{\z_{4H}}}) \xrightarrow{m_3} H^0(\z^2_{4H}(3)|_{D_{\z_{4H}}}).\end{equation}  
The cokernel of $m_3$ is $\cong\mathbb{C}$ by Lemma \ref{gthlem} (1).  Hence now it suffices to find $\mg\in W(2,0,3)$ such that $s_{\mg}$ is not contained in the image of $m_3$.

We choose $Y\in (\p^2)^{[4]}$ consisting of four different points generic enough, and construct a sheaf $\mg$ as in the following sequence
\begin{equation}\label{p2con2}0\ra\mo_{\p^2}(-1)\ra \mg\ra \mi_Y(1)\ra 0.
\end{equation} 
$\mg$ is locally free iff (\ref{p2con2}) is a unit via the identification
$\Ext^1(\mi_Y(1),\mo_{\p^2}(-1))\cong \Ext^2(\mo_Y,\mo_{\p^2}(-1))\cong H^0(\mo_Y)^{\vee}$.  We let $\mg$ be locally free and then by Lemma \ref{p2w3} below $\mg\in W(2,0,3)$.  We claim that $s_{\mg}$ is not contained in the image of $m_3$ in (\ref{p2m3}).  To show the claim, assume $\exists~f_i\in H^0(\mo_{|4H|}(1))$ for $1\leq i\leq 6$, such that $s_{\mg}=\displaystyle{\sum_{i=1}^6}f_i s_{\mg_i}$ where we also write $f_i$ for its pull back via $\pi$.   Since $|4H|\cong \p^{14}$, it is possible to choose a reduced curve $B\in |4H|$ such that $f_i(B)=0$ for all $1\leq i\leq 6$ and moreover $Y\subset B$.  $B$ might be singular at $Y$. For any $y\in Y$ denote by $m_B(y)$ the multiplicity of $B$ at $y$.  Notice that we can choose the curve $B$ smooth at $y$ unless $f_i$ give singularity at $y$.  In other words, if $m_B(y)\geq2$ for all $B\in Z(f_1,\cdots,f_6)$, then $\exists~f_{i_1},f_{i_2},f_{i_3}\in\{f_i\}_{i=1}^6$ such that $f_{i_j}(a)=\overline{f}_{i_j}(a,y),\forall ~a\in |4H|$ and $\overline{f}_{i_j}=\frac{\partial \overline{f}}{\partial x_j}$ for some $\overline{f}\in H^0(\mo_{\p^2}(4)\otimes\mo_{|4H|}(1))$ with $x_1,x_2,x_3$ the homogenous coordinates over $\p^2$.  If $m_B(y)\geq3$ for all $B\in Z(f_1,\cdots,f_6)$, then 
$\{f_i\}_{i=1}^6=\{\frac{\partial^2 \overline{f}}{\partial x_j\partial x_i}(-,y)\}_{1\leq i\leq j\leq 3}$ for some $\overline{f}\in H^0(\mo_{\p^2}(4)\otimes\mo_{|4H|}(1))$. Therefore, since there are 6 linear equations $f_i$, they can at most give singularity with multiplicity 2 at two points or singularity with multiplicity 3 at one point.  Hence we can ask $\sum_{y\in Y}m_{B}(y)\leq 6$.  

Since $s_{\mg}=\displaystyle{\sum_{i=1}^6}f_i s_{\mg_i}$, for all $\mf\in D_{\z_{4H}}$ supported on $B$ we must have $s_{\mg}(\mf)=0$ which is equivalent to that $H^0(\mf\otimes\mg)\neq 0$.  Hence it suffices to find $Z\in B^{[2]}$ such that $H^0(\mi_{Z/B}(1)\otimes\mg)=0$. 

Since $Y$ contains 4 generic points, we may ask $h^0(\mi_Y(2))=2$ and choose another smooth point $x\in B$ such that there is a unique degree 2 curve $C\in |2H|$ which intersects $B$ transversally and contains $Y$ and $x$.  We may also ask $C$ to be integral.  Since $2H.4H=8$ and $\sum_{y\in Y}m_{B}(y)\leq 6$, hence we can find another point $z\in B\cap C$, $z\not\in Y$ and $z\neq x$.  Let $l$ be the line defined by $x$ and $z$.  Let $l\cap B=\{x,x_1,z,z_1\}$.  Then $x_1, z_1\not\in C$ since $2H.H=2$.  Let $Z=\{x,x_1\}$ and $\mf=\mi_{Z/B}(1)$.  Let $\hat{Z}=\{z,z_1\}$.  Then $\mi_{(Z\cup \hat{Z})/B}\cong\mo_B\otimes \mo_{\p^2}(-1)$ since the four points $x,x_1,z,z_1$ are on the same line.  Also by (\ref{p2con2}) we have 
\begin{equation}\label{p2con3} 0\ra \mi_{Z/B}\ra \mg\otimes\mf\ra \mi_Y(1)\otimes \mi_{Z/B}(1)\ra 0.
\end{equation}    
Notice that $Y\subset B$ and $B$ is reduced.  Hence $\mi_Y(1)\otimes \mi_{Z/B}(1)\cong \mo_Y\oplus \mi_{(Y\cup Z)/B}(2)$.  Therefore (\ref{p2con3}) induces the following sequence
\begin{equation}\label{p2con4}0\ra\widetilde{\mi}_{Z-Y}\ra\mg\otimes\mf\ra \mi_{(Y\cup Z)/B}\otimes\mo_{\p^2}(2)\ra0,
\end{equation}
where $\widetilde{\mi}_{Z-Y}$ is the unique torsion-free extension of $\mo_Y$ by $\mi_{Z/B}$ over $B$.  $\widetilde{\mi}_{Z-Y}$ may not be locally free as $B$ may not be smooth at $Y$.  

The dualizing sheaf $\omega_B$ over $B$ is $\mo_B\otimes\mo_{\p^2}(1)$ and the restriction map $H^0(\mo_{\p^2}(n))\ra H^0(\mo_B\otimes\mo_{\p^2}(n))$ is surjective for all $n\geq 1$.  $H^0(\mi_{(Y\cup Z)/B}\otimes\mo_{\p^2}(2))=0$ because the unique degree 2 curve $C$ containing $Y$ and $x$ does not contain $x_1$.  $\Hom(\widetilde{\mi}_{Z-Y},\omega_B)$ is the kernel of the map $\Hom(\mi_{Z/B},\omega_B)\twoheadrightarrow \Ext^1_{B}(\mo_Y,\omega_B)\cong\mo_Y$.  $\Hom(\mi_{Z/B},\omega_B)\cong H^0(\mi_{\hat{Z}/B}\otimes\mo_{\p^2}(2))$ by $\mi_{(Z\cup \hat{Z})/B}\cong\mo_B\otimes \mo_{\p^2}(-1)$.  Since $C$ is integral and contains $Y$ and $z$ and $2H.2H=4$, it is the unique curve of degree 2 containing $Y$ and $z$.  But $z_1\not\in C$, so every non-zero element in $H^0(\mi_{\hat{Z}/B}\otimes\mo_{\p^2}(2))$ does not vanish over $Y$. Thus $\Hom(\widetilde{\mi}_{Z-Y},\omega_B)=0=H^1(\widetilde{\mi}_{Z-Y})^{\vee}$.  Hence $H^0(\widetilde{\mi}_{Z-Y})=0$ since $\chi(\widetilde{\mi}_{Z-Y})=0$.  By (\ref{p2con4}), we have $H^0(\mg\otimes\mf)=0$.  

We have finished the proof of the theorem for $X=\p^2$.  
\end{proof}

\begin{lemma}\label{p2ind}Let $\mf\in D_{\z_{4H}}$.  Let $(Z,C)\in Q$ be the preimage of $\mf$ via the map $g$ as in Lemma \ref{p2z1}.  Let $\mg \in W(2,0,2)$ and be locally free with $C_{\mg}\subset |H|$ the curve consisting of its jumping lines. Then $H^0(\mf\otimes\mg)\neq0\Leftrightarrow l_Z\in C_{\mg}$, where $l_Z$ is the unique line containing $Z$. 
\end{lemma}
\begin{proof}Since $\mg\in W(2,0,2)$ and is locally free, $H^0(\mg(K_X))=0$ and $H^1(\mg(K_X))=H^1(\mg^{\vee})^{\vee}=H^1(\mg)^{\vee}=0$.  Notice that $\mg^{\vee}\cong \mg$ for $det~\mg\cong\mo_X$ and $H^1(\mg)=0$ because $\chi(\mg)=0=h^0(\mg)=h^2(\mg)$.  Hence by (\ref{p2coh1}) we have
\[H^0(\mi_Z(1)\otimes \mg)\cong H^0(\mf\otimes\mg).\]
On the other hand we have 
\begin{equation}\label{gsi}0\ra\mo_X\ra \mi_Z(1)\ra \mo_{l_Z}(-1)\ra0.\end{equation}
Hence $H^0(\mi_Z(1)\otimes\mg)\cong H^0(\mg\otimes \mo_{l_Z}(-1))$.  $H^0(\mg\otimes \mo_{l_Z}(-1))\neq0\Leftrightarrow \mg|_{l_Z}\not\cong\mo_{l_Z}^{\oplus 2}$.  Hence the lemma.
\end{proof}
\begin{lemma}\label{p2w3}For any locally free sheaf $\mg\in W(2,0,3)$, we have the following exact sequence
\begin{equation}\label{p2con}0\ra\mo_{\p^2}(-1)\ra \mg\ra \mi_Y(1)\ra 0,
\end{equation} 
where $Y\in (\p^2)^{[4]}$ and $H^0(\mi_Y(1))=0.$

Conversely, if $\mg$ lies in (\ref{p2con}) and $\mg$ locally free, then $\mg$ is semistable.
\end{lemma}
\begin{proof}$\forall~\mg\in W(2,0,3)$, $\chi(\mg(1))=3>0$, hence there is a nonzero map $\mo_{\p^2}(-1)\ra \mg$ which has to be injective.  So we have 
\[0\ra \mo_{\p^2}(-1)\ra \mg\ra \mk\ra 0,\]
where $\mk$ has to be torsion free because $\mg$ is semistable and locally free.  Hence $\mk\cong \mi_Y(1)$ for some $Y\in (\p^2)^{[4]}$.  $H^0(\mi_Y(1))=0$ since $H^0(\mg)=0$.  

Assume $\mg$ lies in (\ref{p2con}) and not semistable.  Let $\mg_1\subset \mg$ be the rank 1 subsheaf distablizing $\mg$.  Then $c_1(\mg_1)\geq0$, $\mg_1$ is a subsheaf of $\mi_Y(1)$ and (\ref{p2con}) partially split along $\mg_1$.  If $c_1(\mg_1)>0$, then $\mi_Y(1)/\mg_1$ is 0-dimensional and hence $\Ext^1(\mi_Y(1)/\mg_1,\mo_{\p^2}(-1))=0$.  Hence (\ref{p2con}) splits and $\mg$ can not be locally free.  If $c_1(\mg_1)=0$,  then $\mg_1^{\vee\vee}\cong\mo_{\p^2}$ and $\mg$ can not be locally free since $H^0(\mg)=0$.  Hence the lemma.
\end{proof}

\subsection{Proof for $X=\Sigma_e$ with $e\leq 3$ and $L=2G+(e+4)F$.}\label{hg3}\qquad

This case is quite analogous to $X=\p^2$.  However there are still some differences.  
In this subsection $L=2G+(e+4)F$, $L\otimes K_X=2F$, $L^2=16$ and $g_L-1=2$.  Denote by $\mc$ instead of $\mc_{L}$ the universal curve for simplicity.  Since (\ref{comm2}) commutes, Lemma \ref{gthlem} for $X=\Sigma_e$ with $e\leq3$ follows from the following two lemmas. 
\begin{lemma}\label{sez1}The birational map $g:Q\dashrightarrow D_{\z_{L}}$ is a morphism on $Q\setminus R$ with $R$ a closed subset of dimension  $dim~Q-3$.  Moreover $g_{*}\mo_{Q}\cong \mo_{D_{\z_{L}}}$.
\end{lemma}
\begin{lemma}\label{sez2}$\bar{\pi}_{*}\cl_{L}^2\cong\mo_{|L|}^{\oplus 6}\oplus\mo_{|L|}(-2)$.
\end{lemma}
\begin{proof}[Proof of Lemma \ref{sez1}]$\forall~Z\in X^{[2]}$, $h^0(\mi_Z(L\otimes K_X))=1$ or $2$.  We have
\[0\ra \mo_{\p^2}\ra\mi_Z(2F)\ra \mt\ra 0,~\forall~Z\in~X^{[2]}.\]
For $\mt$ either we have
$$0\ra \mo_F(-1)\ra\mt\ra\mo_F(-1)\ra 0,~if~h^0(\mi_Z(2F))=1;$$
or $$0\ra\mo_F\ra\mt\ra\mo_F(-2)\ra 0, ~if~h^0(\mi_Z(2F))=2.$$
Hence $h^1(\mi_Z(L))=0$, $h^0(\mi_Z(L))=L^2-3=13$ and $Q$ is a $\p^{12}$-bundle over $X^{[2]}$ and hence smooth.
Let $\mf$ be in the following sequence
\begin{equation}\label{secoh1}0\ra K_X\xrightarrow{\kappa}\mi_Z(2F)\ra\mf\ra0,\end{equation}
By Lemma \ref{use} $\mf$ is semistable if $h^0(\mi_Z(2F))=1$.  If $h^0(\mi_Z(2F))=2$, then $h^0(\mf)=2$.  If $\mf$ is not semistable, then $\exists ~\mf_1\subset \mf$ such that $\chi(\mf_1)>0$, $H^1(\mf_1)\neq0$ and $h^0(\mf_1)\leq2$.  Hence $\mf_1\cong \omega_C\cong \mo_C\otimes\mo_X(2F)$ with $C\in |L-F|$ ($g_C=2$) and $\omega_C$ the dualizing sheaf of $C$, and $\mf$ lies in the following sequence
\begin{equation}\label{seds1}0\ra\omega_C\ra\mf\ra\mo_{l}(-2)\ra 0,\end{equation}
with $\p^1\cong l\in |F|$.  $\chi(\mo_l(-2),\omega_C)=F.(L-F)=2$ and $\Ext^2(\mo_l(-2),\omega_C)=\Hom(\mo_C,\mo_l(-4))^{\vee}=0$.  Hence all $(\kappa,\mi_Z)$ in (\ref{secoh1}) with $\mf$ in (\ref{seds1}) form a subset $R$ of dimension $\leq dim~|F|+dim~|L-F|+1=13=dim~Q-3$.  Let $\widetilde{R}\subset Q$ contain all $(\kappa,\mi_Z)$ such that $h^0(\mi_Z(2F))=2$.  Then the image of $\widetilde{R}$ inside $X^{[2]}$ is $Sym^2\mc_F$ hence of dimension 3.  Hence $dim~\widetilde{R}=dim~Q-1$.     

Now we only need to show that $D_{\z_L}$ is normal.  $D_{\z_L}$ is Cohen-Macaulay since so is $\ml$.  Hence it is enough to show that $D_{\z_L}$ is smooth in codimension 1.  

Let $D_{\z_L}^s\subset D_{\z_L}$ consist of stable sheaves.  Strictly semistable sheaves in $D_{\z_{L}}$ are $S$-equivalent either to $\mo_F(-1)~\oplus~\mf'$ with $\mf'\in D_{\z_{L-F}}$ or to $\mo_{F}(-1)~\oplus~\mo_F(-1)~\oplus~\mf''$ with $\mf''\in D_{\z_{L-2F}}$.  Hence $dim~(D_{\z_L}\setminus D_{\z_L}^s)= dim~D_{\z_{L-F}}+dim~|F|=13=dim~D_{\z_L}-3$.  If $\mf\in D^s_{\z_L}$, then for any $\mf_1\subsetneq \mf$ we have $\chi(\mf_1)<0$ and hence $H^1(\mf_1)\neq0$.  Hence we always can find a torsion-free extension of $\mf$ by $K_X$, and hence $g$ is surjective on $D^s_{\z_L}$.  $g$ restricted on $Q\setminus \widetilde{R}$ is an isomorphism.  Let $g(\widetilde{R}\setminus R)^s:=g(\widetilde{R}\setminus R)\cap D_{\z_L}^s$.  It is enough to show $dim~g(\widetilde{R}\setminus R)^s\leq dim~D_{\z_L}-2$.  The fiber of $g$ over $g(\widetilde{R}\setminus R)^s$ are of dimension 1 since $\text{ext}^1(\mf,K_X)=2$ for $\mf\in g(\widetilde{R}\setminus R)$, hence $dim~g(\widetilde{R}\setminus R)^s=dim~\widetilde{R}\setminus R-1=dim~Q-2=dim~D_{\z_L}-2$.  

We have proved the lemma. 
\end{proof}
\begin{proof}[Proof of Lemma \ref{sez2}]Analogous to what we did in the proof of Lemma \ref{p2z2}, we have 
\begin{equation}\label{seegl2}\cl_{L}^2\cong \rho^{*} (2F)^{\otimes 2}_{(2)}\otimes \rho^{*}\mo_{X^{[2]}}(-\Delta)\cong \rho^*(2F)^{\otimes2}_{(2)}\otimes \mo_Q(-\Delta_Q).\end{equation}  
where $Q\supset\Delta_{Q}:=\rho^{*}\Delta$.  We have the Hilbert-Chow map $h:Q\ra Sym^2_{|L|}\mc$ which is an isomorphism on $Q\setminus \Delta_{Q}$ and over all $[\{2x\}\subset C]$ such that $C$ are smooth at $x$.  
We have the commutative diagram as follows
\begin{equation}\label{secomm3}\xymatrix{&Q \ar[r]^{\rho}\ar[ld]_{\bar{\pi}}\ar[d]^{h}& X^{[2]}\ar[d]^{h_1}\\ |L|&Sym^2_{|L|}\mc\ar[l]^{\bar{\pi}_1}\ar[r]_{\rho_1}&Sym^2X}.
\end{equation} 
Let $\Delta_{\mc}\subset Sym^2_{|L|}\mc$ be the diagonal.  Then $\Delta_{\mc}\cong \mc$ and $h^*\Delta_{\mc}=\Delta_Q$. 
Notice that $(2F)_{(2)}$ over $X^{[2]}$ is pull back via $h_1$ the line bundle denoted also by $(2F)_{(2)}$ over $Sym^2X$.  Hence by (\ref{seegl2}) we have
\begin{equation}\label{seegl3}\cl_{L}^2\cong h^*(\rho_1^{*} (2F)^{\otimes 2}_{(2)}\otimes \mo_{Sym^2_{|L|}\mc}(-\Delta_{\mc})).\end{equation}  
Define $\ma:=\rho_1^{*} (2F)^{\otimes 2}_{(2)}\otimes \mo_{Sym^2_{|L|}\mc}(-\Delta_{\mc})$.  $Sym^2_{|L|}\mc$ is normal and hence $h_{*}\mo_Q=\mo_{Sym^2_{|L|}\mc}$.  It is enough to show 
\begin{equation}\label{sepl1}(\bar{\pi}_1)_{*}\ma\cong\mo_{|L|}^{\oplus 6}\oplus\mo_{|L|}(-2).\end{equation}

We also have the following commutative diagram
\begin{equation}\label{secom1}\xymatrix{\Delta_{\mc}\ar@{^{(}->}[r] \ar[d]_{\cong}&\mc\times_{|L|}\mc\ar@{^{(}->}[r]\ar@/^2pc/[rrr]_{q_2}\ar@/^3pc/[rrr]^{q_1}\ar[d]_{\sigma} &X\times X\times |L|\ar[r] &X\times X\ar[d]^{\sigma_1}\ar@<0.5ex>[r]^{\widetilde{q}_1}\ar@<-0.5ex>[r]_{\widetilde{q}_2}&X\\ 
\Delta_{\mc}\ar@{^{(}->}[r] &Sym^2_{|L|}\mc\ar[rr]_{\rho_1} & & Sym^2X&.}
\end{equation}
$\sigma_1^{*}(2F)_{(2)}\cong \mo_{X}(2F)^{\boxtimes 2}$.  We then have
\begin{equation}\label{seegl3}\sigma^{*}\ma\cong q_1^{*}\mo_{X}(4F)\otimes q_2^{*}\mo_{X}(4F)\otimes \mo_{\mc\times_{|L|}\mc}(-\Delta_{\mc}).\end{equation}

On $\mc\times_{|L|}\mc$ we have 
\begin{equation}\label{seex1}
0\ra\sigma^*\ma\ra  q_1^{*}\mo_{X}(4F)\otimes q_2^{*}\mo_{X}(4F)\ra  q_1^{*}\mo_{X}(4F)\otimes q_2^{*}\mo_{X}(4F)|_{\Delta_{\mc}}\ra 0.\end{equation}
Define $\bar{\pi}_2=\bar{\pi}_1\circ~ \sigma$.  Then $((\bar{\pi}_2)_{*}(\sigma^*\ma))^{\ks_2}\cong (\bar{\pi}_1)_*\ma$, where $\ks_2$ is the $2^{\text{nd}}$ symmetric group.   

$\ks_2$ acts on $\Delta_{\mc}$ trivially.  $$q_1^{*}\mo_{X}(4F)\otimes q_2^{*}\mo_{X}(4F)|_{\Delta_{\mc}}\cong q^*\mo_{X}(8F)$$ and hence 
\begin{equation}\label{sedel1}(\bar{\pi}_2)_{*}q_1^{*}\mo_{X}(4F)\otimes q_2^{*}\mo_{X}(4F)|_{\Delta_{\mc}}\cong p_{*}(q^*\mo_{X}(8F)),\end{equation}
with $p,q$ the projection of $\mc$ to $|L|$ and $X$ respectively.  

We have on $X\times\ls$
\begin{equation}\label{seuc}0\ra \mo_X(-L)\boxtimes\mo_{|L|}(-1)\ra \mo_{X\times |L|}\ra\mo_{\mc}\ra 0.
\end{equation}
$H^0(\mo_X(8F-L))=0=H^1(\mo_X(8F))$.  Hence we have the following sequence that splits
\begin{equation}\label{sedia}0\ra \mo_{|L|}\otimes H^0(\mo_{X}(8F))\ra p_{*}(q^*\mo_{X}(8F))\ra \mo_{\ls}(-1)\otimes H^1(\mo_X(8F-L))\ra0.
\end{equation}

We then compute $q_1^{*}\mo_{X}(4F)\otimes q_2^{*}\mo_{X}(4F)$ in (\ref{seex1}).  

$\mc\times_{|L|}\mc\subset \mc\times X\subset X\times X\times \ls$.  Hence we have the following exact sequence on $\mc\times X$.  By abuse of notations $p$ is also the projections from $\mc\times X$ to $|L|$, $\widetilde{q}_2$ also the projection from $\mc\times X$ to $X$.
\begin{equation}\label{sepro2}0\ra p^{*}\mo_{|L|}(-1)\otimes \widetilde{q}_2^*\mo_{X}(-L)
\ra\mo_{\mc\times X}\ra\mo_{\mc\times_{\ls}\mc}\ra0.
\end{equation}
Since $H^0(\mo_X(4F-L))=H^1(\mo_X(4F))=0$ and $h^1(\mo_X(4F-L))=1$, by (\ref{sepro2}) we have
\begin{equation}\label{seviw1}{\tiny 0\ra \left.\begin{array}{c}H^0(\mo_X(4F))\\ \otimes \\ p_{*}(q^{*}\mo_{X}(4F))\end{array}\right.\ra(\bar{\pi}_2)_{*}(\left.\begin{array}{c}q_1^{*}\mo_{X}(4F)\\ \otimes \\ q_2^{*}\mo_{X}(4F)\end{array}\right.)\ra \left.\begin{array}{c}\mo_{\ls}(-1)\\ \otimes \\ H^1(\mo_X(4F-L))\\ \otimes\\ p_{*}(q^{*}\mo_X(4F))\end{array}\right.\ra \left.\begin{array}{c}H^0(\mo_X(4F))\\ \otimes \\ R^1p_{*}(q^{*}\mo_{X}(4F))\end{array}\right..}\end{equation} 
Since for every $C\in \ls$, $H^1(\mo_C(4F))=0$ by $H^1(\mo_X(4F))=H^2(\mo_X(4F-L))=0$, we have $R^1p_{*}(q^{*}\mo_{X}(4F))=0$. Hence 
\begin{equation}\label{seviw2}{\small 0\ra \left.\begin{array}{c}H^0(\mo_X(4F))\\ \otimes \\ p_{*}(q^{*}\mo_{X}(4F))\end{array}\right.\ra(\bar{\pi}_2)_{*}(\left.\begin{array}{c}q_1^{*}\mo_{X}(4F)\\ \otimes \\ q_2^{*}\mo_{X}(4F)\end{array}\right.)\ra \left.\begin{array}{c}\mo_{\ls}(-1)\\ \otimes \\ H^1(\mo_X(4F-L))\\ \otimes\\ p_{*}(q^{*}\mo_X(4F))\end{array}\right.\ra0.}\end{equation} 
Using (\ref{seuc}) to compute $p_*(q^*\mo_X(4F))$, we get an exact sequence that has to split as follows.
\begin{equation}\label{seviw3}0\ra \mo_{|L|}\otimes H^0(\mo_{X}(4F))\ra p_{*}(q^*\mo_{X}(4F))\ra \mo_{\ls}(-1)\otimes H^1(\mo_X(4F-L))\ra0.\end{equation}

Combine (\ref{seviw2}) and (\ref{seviw3}) and we have\setlength{\arraycolsep}{0.05cm}
\begin{equation}\label{seviw4}\renewcommand{\arraystretch}{1.2}
{\scriptsize\begin{array}[c]{ccccccccc}&&0&&0&&0&& 
\\ &&\Big\downarrow&&\Big\downarrow&&\Big\downarrow&&\\
0&\longrightarrow&\mo_{\ls}\otimes H^0(\mo_X(4F))^{\otimes 2}&\longrightarrow
& \me_1 &\longrightarrow &\mo_{\ls}(-1)\otimes (\left.\begin{array}{c}H^0(\mo_X(4F))\\ \otimes \\ H^1(\mo_X(4F-L))\end{array}\right.)&\longrightarrow
&0\\
&&\Big\downarrow&&\Big\downarrow&&\Big\downarrow&&\\
0&\longrightarrow& \left.\begin{array}{c}H^0(\mo_X(4F))\\ \otimes \\ p_{*}(q^{*}\mo_{X}(4F))\end{array}\right.&\longrightarrow
 &(\bar{\pi}_2)_{*}(\left.\begin{array}{c}q_1^{*}\mo_{X}(4F)\\ \otimes \\ q_2^{*}\mo_{X}(4F)\end{array}\right.)&\longrightarrow &\left.\begin{array}{c}\mo_{\ls}(-1)\\ \otimes \\ H^1(\mo_X(4F-L))\\ \otimes\\ p_{*}(q^{*}\mo_X(4F))\end{array}\right.&\longrightarrow &0\\
&&\Big\downarrow&&\Big\downarrow&&\Big\downarrow&&\\
 0&\longrightarrow&  \mo_{\ls}(-1)\otimes (\left.\begin{array}{c}H^0(\mo_X(4F))\\ \otimes \\ H^1(\mo_X(4F-L))\end{array}\right.)&\longrightarrow & \me_2 &\longrightarrow &\mo_{\ls}(-2)\otimes H^1(\mo_X(4F-L))^{\otimes 2}&\longrightarrow &0\\
 &&\Big\downarrow&&\Big\downarrow&&\Big\downarrow&&
\\ &&0&&0&&0&&
\end{array}}
\end{equation}
All the sequences in (\ref{seviw4}) split and hence 
\begin{equation}\label{sesta1} {\scriptsize q_1^{*}\mo_{X}(4F)~\otimes ~q_2^{*}\mo_{X}(4F)\cong\renewcommand{\arraystretch}{1.5}\begin{array}{c}\mo_{\ls}\otimes H^0(\mo_X(4F))^{\otimes2}\\ \bigoplus \\\mo_{\ls}(-1)\otimes\left(\begin{array}{c}H^0(\mo_X(4F))\\ \otimes \\ H^1(\mo_X(4F-L))\end{array}\right)^{\oplus 2}\\ \bigoplus \\ \mo_{\ls}(-2)\otimes H^1(\mo_X(4F-L))^{\otimes 2}\end{array}}.\end{equation}
Easy to see how $\ks_2$ acts on the right hand side of (\ref{sesta1}) and hence
\begin{equation}\label{sesta2}{\scriptsize (q_1^{*}\mo_{X}(4F)\otimes q_2^{*}\mo_{X}(4F))^{\ks_2}\cong\renewcommand{\arraystretch}{1.5}\begin{array}{c}\mo_{\ls}\otimes S^2H^0(\mo_X(4F))\\ \bigoplus \\\mo_{\ls}(-1)\otimes\left(\begin{array}{c}H^0(\mo_X(4F))\\ \otimes \\ H^1(\mo_X(4F-L))\end{array}\right)\\ \bigoplus \\ \mo_{\ls}(-2)\otimes S^2H^1(\mo_X(4F-L))\end{array}}.\end{equation}
 
By (\ref{sedel1}) and (\ref{sedia}), we have 
\begin{eqnarray}\label{sedel2}((\bar{\pi}_2)_{*}q_1^{*}\mo_{X}(4F)\otimes q_2^{*}\mo_{X}(4F)|_{\Delta_{\mc}})^{\ks_2}&\cong&
(\bar{\pi}_2)_{*}q_1^{*}\mo_{X}(4F)\otimes q_2^{*}\mo_{X}(4F)|_{\Delta_{\mc}}\nonumber \\ 
\cong p_{*}(q^{*}\mo_X(8F))&\cong&\renewcommand{\arraystretch}{1.5} {\scriptsize\left.\begin{array}{c}\mo_{\ls}\otimes H^0(\mo_X(8F))\\ \bigoplus \\
\mo_{\ls}(-1)\otimes H^1(\mo_X(8F-L))\end{array}\right.}.\end{eqnarray} 
Therefore by (\ref{seex1}), (\ref{sesta2}) and (\ref{sedel2}), we have the following sequence 
\begin{equation}\label{sepf}\renewcommand{\arraystretch}{1.5}{\scriptsize 0\ra ((\bar{\pi}_2)_{*}(\sigma^*\ma))^{\ks_2}\ra \begin{array}{c}\mo_{\ls}\otimes S^2H^0(\mo_X(4F))\\ \bigoplus \\\mo_{\ls}(-1)\otimes\left(\begin{array}{c}H^0(\mo_X(4F))\\ \otimes \\ H^1(\mo_X(4F-L))\end{array}\right)\\ \bigoplus \\ \mo_{\ls}(-2)\otimes S^2H^1(\mo_X(4F-L))\end{array}\xrightarrow{\mathtt{r}} \left.\begin{array}{c}\mo_{\ls}\otimes H^0(\mo_X(8F))\\ \bigoplus \\
\mo_{\ls}(-1)\otimes H^1(\mo_X(8F-L))\end{array}\right..}\end{equation}
We want the map $\mathtt{r}$ in (\ref{sepf}) to be surjective.  
Notice that the restriction $\mathtt{r}|_{\mo_{\ls}\otimes S^2H^0(\mo_X(4F))}:\mo_{\ls}\otimes S^2H^0(\mo_X(4F))\ra\mo_{\ls}\otimes H^0(\mo_X(8F))$ is given 
by the multiplication map $S^2H^0(\mo_X(4F))\ra H^0(\mo_X(8F))$, $(s_1,s_2)\mapsto s_1\cdot s_2$.  Hence $\mathtt{r}$ is surjective on $\mo_{\ls}\otimes H^0(\mo_X(8F))$.  

For any curve $C\in \ls$, $H^1(\mo_X(8F-L))\cong H^0(\mo_{C}(8F))/H^0(\mo_X(8F))$ and $H^1(\mo_X(4F-L))\cong H^0(\mo_{C}(4F))/H^0(\mo_X(4F))$.
The map 
$\mathtt{r}$ sends $\mo_{\ls}(-1)\otimes H^0(\mo_X(4F))\otimes H^0(\mo_X(4F-L))$ to $\mo_{\ls}(-1)\otimes H^1(\mo_X(8F-L))$ and its restriction on $\mo_{\ls}(-1)\otimes H^0(\mo_X(4F))\otimes H^0(\mo_X(4F-L))$ is given by the following multiplication
\[H^0(\mo_X(4F))\otimes H^0(\mo_C(4F))/H^0(\mo_X(4F))\xrightarrow{m_{\mathtt{r}}} H^0(\mo_C(8F))/H^0(\mo_X(8F)),\]
where we identify $H^1(\mo_X(nF-L))$ with $H^0(\mo_C(nF))/H^0(\mo_X(nF))$ for $n=4,8$.

Let $\Delta_X\subset X\times X$ be the diagonal and $\mi_{\Delta_X}$ be its ideal sheaf.  We have  
\begin{equation}\label{sesur1}{\small 0\ra \mi_{\Delta_X} \otimes\left(\begin{array}{c}\mo_X(4F)\\ \boxtimes\\ \mo_X(4F-L)\end{array}\right)\ra\begin{array}{c}\mo_X(4F)\\ \boxtimes \\ \mo_X(4F-L)\end{array}\ra
\left.\begin{array}{c}\mo_X(4F)\\ \boxtimes\\ \mo_X(4F-L)\end{array}\right|_{\Delta_X}\ra 0.}\end{equation}
$\mo_X(4F)\boxtimes\mo_X(4F-L)|_{\Delta_X}\cong \mo_X(8F-L)$, $H^1(\mo_X(4F)\boxtimes\mo_X(4F-L))\cong H^0(\mo_X(4F))\otimes H^1(\mo_X(4F-L))$ and the map $m_{\mathtt{r}}$ is actually given by the following restriction induced by (\ref{sesur1})
\begin{equation}\label{sesur2}H^1(\mo_X(4F)\boxtimes\mo_X(4F-L))\ra H^1((\mo_X(4F)\boxtimes\mo_X(4F-L))|_{\Delta_X})
\end{equation}

By Lemma \ref{sesur} below we have $m_{\mathtt{r}}$ is surjective and hence so is $\mathtt{r}$.  Therefore by (\ref{sepf})
\[(\bar{\pi}_1)_*\ma\cong((\bar{\pi}_2)_{*}(\sigma^*\ma))^{\ks_2}\cong \mo_{\ls}^{\oplus 6}\oplus\mo_{\ls}(-2).\]
Notice that $H^0(\mo_X(8F-L))=0$, $H^2(\mo_X(8F-L))=H^0(\mo_X(-6F))^{\vee}=0$.  Hence $h^1(\mo_X(8F-L))=-\chi(\mo_X(8F-L))=5$.  Also $h^1(\mo_X(4F-L))=-\chi(\mo_X(4F-L))=1$, and $h^0(\mo_X(nF))=n+1$ for all $n\geq1$.  The lemma is proved.
\end{proof}
\begin{lemma}\label{sesur}The restriction map in (\ref{sesur2}) is surjective.
\end{lemma}
\begin{proof}$X$ is a ruled surface with projection $\tau:X\ra\p^1$.  $X\times_{\p^1}X$ is a divisor in $X\times X$ associated to the line bundle $\mo_X(F)\boxtimes\mo_X(F)$. $H^2(\mo_X(3F)\boxtimes\mo_X(3F-L))\cong H^0(\mo_X(3F))\otimes H^2(\mo_X(3F-L))=0$.  Hence the following map is surjective
\begin{equation}\label{sesur3}H^1(\mo_X(4F)\boxtimes\mo_X(4F-L))\twoheadrightarrow H^1((\mo_X(4F)\boxtimes\mo_X(4F-L))|_{X\times_{\p^1}X})
\end{equation}

$\Delta_{X}$ is a divisor on $X\times_{\p^1}X$ associated to the line bundle $(\mo_{X}(G)\boxtimes\mo_X(G))\otimes \tau^{*}\mo_{\p^1}(e)$.  This is because $\mo_{X\times_{\p^1}X}(\Delta_X)|_{\mo_{\Delta_X}}\cong \mt_{X/\p^1}\cong \mo_X(2G+eF) $ and $\mo_{X\times_{\p^1}X}(\Delta_X)$ restricted to each fiber $F_{\tau}\cong\p^1\times\p^1$ of $\tau$ is $\mo_{\p^1}(1)\boxtimes\mo_{\p^1}(1)$.   
{\footnotesize $$(\mo_X(4F)~\boxtimes~\mo_X(4F-L))|_{X\times_{\p^1} X}~\otimes ~\mi_{\Delta_X/X\times_{\p^1}X} \cong 
(\mo_{X}(-G)~\boxtimes~\mo_X(-3G))~\otimes~ \tau^{*}\mo_{\p^1}(4-2e).$$}
For each fiber $F_{\tau}$ of $\tau$, we have $$(\mo_{X}(-G)~\boxtimes~\mo_X(-3G))|_{F_{\tau}}\cong\mo_{\p^1}(-1)~\boxtimes~\mo_{\p^1}(-3)$$ and hence $R^i\tau_{*}(\mo_{X}(-G)~\boxtimes~\mo_X(-3G))=0$ for all $i$.  Therefore 
$$H^2((\mo_X(4F)~\boxtimes~\mo_X(4F-L))|_{X\times_{\p^1} X}~\otimes ~\mi_{\Delta_X/X\times_{\p^1}X})=0.$$
Hence the following map is surjective
\begin{equation}\label{sesur4}H^1((\mo_X(4F)\boxtimes\mo_X(4F-L))|_{X\times_{\p^1}X})\twoheadrightarrow H^1((\mo_X(4F)\boxtimes\mo_X(4F-L))|_{\Delta_X})
\end{equation}
The map in (\ref{sesur2}) is obtained by composing maps in (\ref{sesur3}) and (\ref{sesur4}).  Hence the lemma.
\end{proof}

\begin{proof}[Proof of Theorem \ref{gth3} for $X=\Sigma_e$.]
We see that $-K_X$ is ample iff $e\leq1$.  Analogously to the proof for $X=\p^2$ in \S \ref{p2g3}, we will at first find $\mg_i\in W(2,0,2)$, $1\leq i\leq 6$ such that $\{s_{\mg_i}\}_{i=1}^6$ are linearly independent restricted to $D_{\z_{L}}$, and then we will find $\mg\in W(2,0,4)$ such that $s_{\mg}$ is not contained in the image of the multiplication map
\begin{equation}\label{sem3}H^0(\pi^{*}\mo_{|L|}(2))\otimes H^0(\z^2_{L}(2)|_{D_{\z_{L}}}) \xrightarrow{m_4} H^0(\z^2_{L}(4)|_{D_{\z_{L}}}).\end{equation}  

Choose four distinct points $\{x_1,x_2,x_3,x_4\}\subset X$ such that any two of them do not lie on the same fiber or on a curve in $|G|$.  Denote by $l_i$ the unique fiber containing $x_i$.  Then $l_i\neq l_j$ for $i\neq j$.  Denote by $\mi_{ij}~(i<j)$ the ideal sheaf of $\{x_i,x_j\}$, then $H^0(\mi_{ij}(F))=H^0(\mi_{ij}(G))=0$.  Construct $\mg_{ij}~(i<j)$ as a locally free extension of $\mi_{ij}(F)$ by $\mo_X(-F)$ as follows.
\begin{equation}\label{secon1}0\ra\mo_X(-F)\ra\mg_{ij}\ra\mi_{ij}(F)\ra0.
\end{equation}
Then by Lemma 6.27 in \cite{GY}, $\mg_{ij}$ are slop-stable.  The number of $\mg_{ij}$ is 6.  To see the linear independence of $s_{\mg_{ij}}|_{D_{\z_L}}$, we define $\mf_{ij}:=\mo_C\oplus\mo_{l_i}(-1)\oplus\mo_{l_j}(-1)$ where $C\in |-K_X|$.  Then $\mf_{ij}\in D_{\z_L}$ and $H^0(\mg_{ij}\otimes \mf_{k\ell})\neq 0\Leftrightarrow \{x_i,x_j\}\cap \{x_k,x_{\ell}\}\neq\emptyset$ (see the proof of Lemma 6.29 in \cite{GY}).  Hence $s_{\mg_{ij}}(\mf_{k\ell})\neq 0\Leftrightarrow \{i,j,k,\ell\}=\{1,2,3,4\}$.  Hence $s_{\mg_{ij}}|_{D_{\z_L}}$ can not be linearly dependent.

Let $Y:=\{x_1,x_2,x_3,x_4\}\in X^{[4]}$, then $H^0(\mi_Y(3F))=H^0(\mi_Y(3G))=0$.  We can ask moreover $H^0(\mi_Y(F+G))=0$.  We construct a locally free sheaf $\mg\in W(2,0,4)$ as in the following sequence
\begin{equation}\label{secon2}0\ra\mo_{X}(-2F)\ra \mg\ra \mi_Y(2F)\ra 0.
\end{equation} 
This can be done by Lemma \ref{sew3} below.  Then $\mg|_{l_{ij}}\cong \mo_{\p^1}(-1)\oplus\mo_{\p^1}(1)$ and hence $s_{\mg}(\mo_C\oplus\mo_{l_i}(-1)\oplus\mo_{l_j}(-1))=0$ for all $C\in|-K_X|$ and $1\leq i<j\leq 6$.  If $s_{\mg}$ is contained in the image of $m_4$ in (\ref{sem3}), then $\exists~f_{ij}\in H^0(\mo_{|L|}(2))$ for $1\leq i<j\leq 4$, such that $s_{\mg}=\displaystyle{\sum_{1\leq i<j\leq 4}}f_{ij} s_{\mg_{ij}}$ where we also write $f_{ij}$ for its pull back via $\pi$.  Then $f_{ij}$ vanishes over the image of $\imath_{ij}:|-K_X|\hookrightarrow\ls$ where $\imath_{ij}$ is given by $C\mapsto C\cup l_k\cup l_{\ell}$ such that $\{i,j,k,\ell\}=\{1,2,3,4\}$.  However the image of $\imath_{ij}$ is defined by linear equations in $\ls$ and hence $f_{ij}$ has to be a product of two linear equations.  Write $f_{ij}=g_{ij}\cdot h_{ij}$ such that $g_{ij},h_{ij}\in H^0(\mo_{\ls}(1))$. 

Since $|L|\cong \p^{14}$, it is possible to choose a curve $B\in |L|$ such that $g_{ij}(B)=0$ for all $1\leq i<j\leq 4$ and moreover $Y\subset B$.  Since $s_{\mg}=\displaystyle{\sum_{1\leq i<j\leq 4}}g_{ij}h_{ij} s_{\mg_{ij}}$, for all $\mf\in D_{\z_{4H}}$ supported on $B$ we must have $s_{\mg}(\mf)=0$ which is equivalent to that $H^0(\mf\otimes\mg)\neq 0$.  Hence it suffices to find $Z\in B^{[2]}$ such that $H^0(\mi_{Z/B}(2F)\otimes\mg)=0$.   

By Lemma \ref{sez1}, we see that for any $\widetilde{Z}\in B^{[2]}$, $\mi_{\widetilde{Z}/B}(2F)$ is semistable iff $\widetilde{Z}$ is not contained in a fiber component of $B$.  Choose a generic fiber $l_5$ different from $l_i$ for $1\leq i\leq4$ such that $l_5\cap B=\{y,z\}$.  Let $Z=\{y,z\}\in B^{[2]}$, then $\mf:=\mi_{Z/B}(2F)\cong\mo_B(F)\in D_{\z_L}$.  Then by (\ref{secon2}) we have
\begin{equation}\label{secon3} 0\ra \mo_B(-F)\ra \mg\otimes\mf\ra \mi_Y(3F)\otimes \mo_{B}\ra 0.
\end{equation}    
Notice that $Y\subset B$.  Hence $\mi_Y(3F)\otimes \mo_B\cong \mo_{Y}\oplus \mi_{Y/B}(3F)$.  Therefore (\ref{secon3}) induces the following sequence
\begin{equation}\label{secon4}0\ra\mi_{Y/B}^{\vee}(-F)\ra\mg\otimes\mf\ra \mi_{Y/B}(3F)\ra0,\end{equation}
where $\mi_{Y/B}^{\vee}$ is the unque torsion free extension of $\mo_B$ by $\mo_Y$ over $B$.

The dualizing sheaf $\omega_B$ over $B$ is $\mo_B\otimes\mo_{X}(2F)$ and the restriction map $H^0(\mo_{X}(nF))\ra H^0(\mo_B\otimes\mo_{X}(nF))$ is surjective for all $0\leq n\leq 3$.  $H^0(\mi_{Y/B}(3F))=0$ because $H^0(\mi_Y(3F))=0$.  $\Hom(\mi_{Y/B}^{\vee}(-F),\omega_B)$ is the kernel of the map $H^0(\mo_B(3F))\twoheadrightarrow \Ext^1_B(\mo_Y,\mo_{B}(3F))\cong \mo_Y$.  But points in $Y$ lie on 4 different fibers and hence every non-zero element in $H^0(\mo_B(3F))$ does not vanish on $Y$ and hence $\Hom(\mi_{Y/B}^{\vee}(-F),\omega_B)\cong H^1(\mi_{Y/B}^{\vee}(-F))^{\vee}=0$.
Therefore we have $H^0(\mi_{Y/B}^{\vee}(-F))=0$ since $\chi(\mi_{Y/B}^{\vee}(-F))=0$.  By (\ref{secon4}), we have $H^0(\mg\otimes\mf)=0$.  

We have finished the proof of the whole theorem.  
\end{proof}

\begin{lemma}\label{sew3}Assume $H=G+aF$ for some $a\in\mathbb{Q}$ and $a>\max\{e,1\}$.  Let $\mg$ be locally free and lie in the following sequence
\begin{equation}\label{secon}0\ra\mo_{X}(-2F)\ra \mg\ra \mi_Y(2F)\ra 0,
\end{equation} 
where $Y\in X^{[4]}$ and $H^0(\mi_Y(3F))=H^0(\mi_Y(3G))=H^0(\mi_Y(G+F))=0.$ Then $\mg\in W(2,0,4)$ and slop-stable .
In particular, we can always find a locally free $\mg$ in (\ref{secon})
\end{lemma}
\begin{proof}The proof is analogous to Lemma 6.27 in \cite{GY}.  Since $\mg$ is locally free, to show $\mg$ is slop-stable it is enough to show $H^0(\mg\otimes \mo_X(P))=0$ for any $P\in \Pic(X)$ and $P.H\leq 0$.  By (\ref{secon}), It is enough to show $H^0(\mo_X(-2F+P))=H^0(\mi_Y(2F))=0$ for all $P.H\leq 0$.  It is obvious that $H^0(\mo_X(-2F+P))=0$ since $(-2F+P).H<0$.  If $H^0(\mi_Y(2F+P))\neq 0$, then $H^0(\mo_X(2F+P))\neq 0$.  On the other hand, $(2F+P).H\leq 2$.  If $e=0$, then $2F+P=2G,2F$ or $G+F$ and $H^0(\mi_Y(2F+P))=0$.  If $e\neq0$, then $H^0(\mo_X(nG))\cong H^0(\mo_X(G))\cong \mathbb{C}$ for all $n\geq1$ and $H^0(\mo_X(F+nG))\cong H^0(\mo_X(F+G))$ for all $n\geq 1$.  $2F+P=2F$, $mG$ with $m\leq \frac2{a-e}$ or $F+nG$ with $n\leq\frac1{a-e}$.  Hence $H^0(\mi_Y(F+P))=0$ and hence $\mg$ is slop stable.  

Cayley-Bacharach condition is fulfilled by $H^0(K_X(4F))=0$, hence we can always find a locally free $\mg$ in (\ref{secon}).
\end{proof}
\begin{flushleft}{\textbf{Acknowledgments.}} I would like to thank T. Abe for inviting me to a mini-workshop at RIMS in Kyoto, where I started to write the paper. 
\end{flushleft}

\end{document}